\documentclass[bj,numbers]{imsart}

\usepackage{hyperref}
\usepackage[section]{algorithm}
\usepackage{mathrsfs}
\usepackage{mathtools}
\usepackage{booktabs}
\PassOptionsToPackage{shortlabels}{enumitem}
\usepackage[table]{xcolor}
\usepackage{amsmath, bbm, amssymb, algorithmic, multirow, graphicx, array, pifont, makecell,cancel}

\startlocaldefs
\numberwithin{equation}{section}
\theoremstyle{plain}
\newtheorem{theorem}{Theorem}
\newtheorem{proposition}{Proposition}
\newtheorem{lemma}{Lemma}
\theoremstyle{definition}

\definecolorset{HTML}{}{}{%
    light green,C7EBBA;%
    brick,E86A58
}

\endlocaldefs

\begin{document}

\begin{frontmatter}
\title{Large deviations for Independent Metropolis Hastings and Metropolis-adjusted Langevin algorithm}
\runtitle{Large deviations for IMH and MALA}

\begin{aug}
\author[A]{\inits{FM}\fnms{Federica}~\snm{Milinanni}\ead[label=e1]{federica\_milinanni@brown.edu}}
\author[B]{\inits{PN}\fnms{Pierre}~\snm{Nyquist}\ead[label=e2]{pnyquist@chalmers.se}}
\address[A]{Institute for Computational and Experimental Research in
Mathematics (ICERM), Brown University, Providence, RI 02903, USA\printead[presep={,\ }]{e1}}

\address[B]{Department of Mathematical Sciences, Chalmers University of Technology and University of Gothenburg, 412 96 Gothenburg, Sweden\printead[presep={,\ }]{e2}}
\end{aug}

\begin{abstract}
In this paper, we prove large deviation principles for the empirical measures associated with the Independent Metropolis Hastings (IMH) sampler and  the Metropolis-adjusted Langevin Algorithm (MALA). These are the first large deviation results for empirical measures of Markov chains arising from specific Metropolis-Hastings methods on a continuous state space. Moreover, we show that the existing large deviation framework, that we developed in a previous work (Milinanni and Nyquist, 2024) does not cover the Random Walk Metropolis sampler, even in cases when the underlying Markov chain is geometrically ergodic.
\end{abstract}


\end{frontmatter}


\section{Introduction}
\label{sec:intro}
Markov chain Monte Carlo (MCMC) methods have become ubiquitous across scientific disciplines as the standard tool for sampling from a given probability distribution. The canonical MCMC method, and the central building block for many of the modern methods now being used, is the Metropolis-Hastings (MH) algorithm \cite{metropolis1953equation, hastings1970monte}; for formulations on more general state spaces, see, e.g., \cite{Tierney98, G-HKM23}. Examples of some of the most popular methods built on the Metropolis-Hastings algorithm are the Independent Metropolis-Hastings (IMH) algorithm  \cite{Tie94, robert2004monte}, the Random Walk Metropolis (RWM) algorithm  \cite{mengersen1996rates} and the Metropolis-adjusted Langevin algorithm (MALA) \cite{besag1994comments, roberts1996exponential, roberts1998optimal}.

When trying to sample from a target measure $\pi$ using MCMC, using a standard method, such as Metropolis-Hastings, can come with slow convergence and/or great computational cost for a desired accuracy. Performance analysis of MCMC methods has therefore become an important topic at the intersection of applied probability and (computational) statistics. A first step towards understanding the performance of a given MCMC method is an analysis of the convergence to the target $\pi$; classical tools include the spectral gap of the associated dynamics, mixing times, asymptotic variance and functional inequalities (Poincar\'e, log-Sobolev); see, e.g., \cite{bedard2008optimal, rosenthal2003asymptotic, DHN00, franke2010behavior, frigessi1993convergence}. Some important results on properties and performance of Metropolis-Hastings algorithms are given in \cite{mengersen1996rates, roberts1996exponential, roberts1996geometric, GGR97, roberts1997geometric, RR01, CRR05, G-HKM23}; see also the references therein. Adding to these classical results, and overall to the toolbox for analysing MCMC methods, in a series of recent papers \cite{ALP+22a, ALP+23a, PSW24}, Andrieu, Lee, Power and Wang and co-authors use weak Poincar\'e inequalities to study convergence of discrete time Markov chains, aiming specifically at analysing various MCMC algorithms. 

An important quantity in the study of MCMC methods, and for comparing their performance, is the convergence rate of time averages. At the heart of MCMC is the property that for an ergodic chain $\{ X_i \} _{i \geq 0}$ with invariant distribution $\pi$ and an observable $f \in L ^1 (\pi)$, the $n$-step averages $\frac{1}{n} \sum _{i=0} ^{n-1} f(X_i)$ can be used to approximate $\mathbb{E} _\pi [f(X)]$. Such averages can in turn be viewed as integrals of $f$ with respect to the empirical measures $\frac{1}{n} \sum _{i=0} ^{n-1}\delta_{X_i}(\cdot)$ of the Markov chain. The convergence of time averages is therefore intrinsically linked to the convergence of the empirical measures to $\pi$, and the latter can thus be used to understand the performance of MCMC methods.

Starting with the study of parallel tempering in \cite{PDD+11, dupuis2012infinite}, we have seen an increased interest in using the theory of large deviations for empirical measures to study MCMC methods. In addition to the original work on parallel tempering/infinite swapping by Doll, Dupuis and co-authors, the following are some notable examples of using empirical measure large deviations in the MCMC setting: in \cite{rey2015irreversible, rey2016improving} Rey-Bellet and Spiliopoulos study the convergence properties of certain reversible and irreversible Markov processes using empirical measure large deviations, showing improved performance of non-reversible methods by probing the associated rate function; in \cite{bierkens2016non} Bierkens considers a continuous-time Metropolis-Hastings algorithm on a finite state space and uses large deviations to prove improved convergence for an irreversible version; for parallel tempering and infinite swapping, in \cite{doll2018large} Doll, Dupuis and Nyquist use large deviation rate functions, combined with associated stochastic control problems, to analyse the convergence properties of the methods; in \cite{bierkens2021large} Bierkens, Nyquist and Schlottke use empirical measure large deviations to study the zig-zag process, obtaining the first optimality result for the so-called switching rate in one dimension; in \cite{DW22} Dupuis and Wu use empirical measure large deviations to solve a long-standing open problem on temperature selection for parallel tempering and infinite swapping in the low-temperature regime.

In our recent work \cite{milinanni2024large}, with a performance analysis of methods built on the Metropolis-Hastings mechanism in mind, we derive a large deviation principle (LDP) for the empirical measures of Metropolis-Hastings chains on a continuous state space; the results hold for a general collection of discrete-time Markov processes whose transition kernels take on a particular form, with Metropolis-Hasting as a special case. This is the first large deviation result that covers Metropolis-Hastings chains on general state spaces, and \cite{milinanni2024large} therefore opens the possibility for new analysis and insights into Metropolis-Hastings-type methods. The results are of a general type, in that they hold for any choice of proposal distribution that satisfies certain assumptions; the critical assumption is the existence of a particular type of Lyapunov function (see Assumption \ref{ass:compactSpaceOrLyapunov} in Section \ref{sec:LDP}). However, establishing whether or not such a function exists for specific choices of proposal distribution is a highly challenging task. Therefore, despite the advances made in \cite{milinanni2024large}, for specific choices of proposal distribution, i.e., for specific MCMC methods, whether or not an LDP holds remained an open question.

In this work we answer the question for IMH and MALA, and partly for RWM, covering three of the most common MCMC methods based on the Metropolis-Hastings algorithm. The three methods are staples within MCMC and there is a vast literature on their properties and use in various applications—the following are some of the most relevant existing (theoretical) results for the three methods: \cite{mengersen1996rates, AP07, Wang22, BJ24} (IMH); \cite{roberts1996exponential, roberts1998optimal, RR01, CRR05, BLM14} (MALA); \cite{mengersen1996rates, roberts1996geometric, JH2000, GGR97, MPS12, HSV14, ALP+22a} (RWM).

We show that for IMH and MALA, under conditions that ensure the corresponding Markov chain is geometrically ergodic (except for a boundary case for MALA, see Section \ref{sec:conjecture}), an LDP holds for the underlying empirical measures. The main results are Theorems \ref{thm:LDP_indep} and \ref{thm:LDP_MALA}, stating an LDP for IMH and MALA, respectively. These are the first LDPs for particular instances of Metropolis-Hastings algorithms when the state space is a (uncountable) subset of $\mathbb{R}^d$; for the IMH, the proof is rather straightforward, whereas for MALA it becomes rather involved to show the existence of a suitable Lyapunov function (via the conditions of Lemma \ref{lem:rto0}). Moreover, in Section \ref{sec:RWM} we show that the large deviation framework for MH chains developed in \cite{milinanni2024large} does not cover RMW---the main result in this direction is Proposition \ref{prop:RWMnoLyapunov}, stating that the type of Lyapunov function needed for the results in \cite{milinanni2024large} to apply cannot exist for any RWM algorithm, even when the underlying chain is geometrically ergodic.

The combination of the current paper and the results in \cite{milinanni2024large} marks the first step towards a more general theory of large deviations for Metropolis-Hastings chains. Future work includes extending our results to more general state spaces, as well as more general conditions and, by extension, (other) specific proposal distributions.

The remainder of the paper is organised as follows. In Section \ref{sec:preliminaries} we provide the preliminaries needed for the subsequent sections: notation and definitions (Section \ref{sec:notationDef}); a description of the Metropolis-Hastings algorithm (Section \ref{sec:MH}); an overview of large deviations for empirical measures, including the results 
of \cite{milinanni2024large} (Section \ref{sec:LDP}) and a brief remark on how large deviation results can be used in the analysis of MCMC methods. In Section \ref{sec:propLyapRd}, we take a closer look at Assumption \ref{ass:compactSpaceOrLyapunov} and obtain an equivalent version of Part (b) of this assumption (see Lemma \ref{lem:rto0}), tailored to the MCMC setting. Section \ref{sec:main} then contains the main theoretical results of the paper: we consider large deviations for the empirical measures of Markov chains arising from the three common MCMC methods IMH (Section \ref{sec:IMH}), MALA (Section \ref{sec:MALA}) and RWM (Section \ref{sec:RWM}). Based on the results in Section \ref{sec:main}, and previous large deviation work for Markov chains, such as \cite{KM03, KM05}, we give a brief summary and discussion about LDPs for Metropolis-Hastings chains and geometric ergodicity in Section \ref{sec:conjecture}. The paper ends with technical proofs of the main results in Section \ref{sec:proofs}, with the longer proofs for MALA provided in the \hyperref[appn]{Appendix}.

\section{Preliminaries}
\label{sec:preliminaries}
\subsection{Notation and definitions}
\label{sec:notationDef}
Throughout the paper we work with some probability space $(\Omega, \mathcal{F}, \mathbb{P})$. The state space of the stochastic processes under consideration is the $d$-dimensional Euclidean space $\mathbb{R}^d$ for $d\in\mathbb{N}$, which we will often denote by $S$ for ease of notation. For $x,y\in S$, we denote by $\langle x,y\rangle$ the scalar product between the two vectors, and $|x|=\sqrt{\langle x,x\rangle}$ is the Euclidean norm of $x$.

Given a set $A\subseteq S$, let $-A=\{x\in S\,:\,-x\in A\}$ and let $A^\circ$ be the interior of $A$. We denote by $x\mapsto I\{x\in A\}$ the \textit{indicator function} of $A$. 

We denote by $\mathcal{P}(S)$ the space of probability measures on $S$, and by $\mathcal{B}(S)$ the Borel $\sigma$-algebra on $S$. We endow $\mathcal{P}(S)$ with the topology corresponding to the weak convergence of probability measures.
If not otherwise specified, \textit{almost all} and \textit{almost surely} refer to the Lebesgue measure on $S$, which is denoted by $\lambda$; for integration with respect to $\lambda$, we use the standard notation $dx$ for $\lambda (dx)$. For a probability measure $\mu(dx)\in\mathcal{P}(S)$ that is absolutely continuous with respect to $\lambda$ ($\mu\ll\lambda$), we abuse notation slightly and denote by $\mu(x)$ its density with respect to $\lambda$, so that $\mu(dx)=\mu(x)dx$. Given $\gamma\in\mathcal{P}(S^2)$, let $[\gamma]_1$ and $[\gamma]_2$ denote the first and second marginal of $\gamma$, respectively. For $\mu\in\mathcal{P}(S)$, define
\begin{align}
\label{eq:defA}
    A(\mu)=\{\gamma\in\mathcal{P}(S^2)\,:\,[\gamma]_1=[\gamma]_2=\mu\}.
\end{align}

For probability measures $\mu,\nu\in\mathcal{P}(S)$, we denote the \textit{total variation distance} between $\mu$ and $\nu$ by $\lVert \mu-\nu\rVert_{TV}=\sup_{A\in\mathcal{B}(S)}|\mu(A)-\nu(A)|$.
For $\nu \in \mathcal{P}$, the \textit{relative entropy} (with respect to $\nu$) is defined as the map $R(\cdot \parallel \nu) : \mathcal{P}(S) \to [0, \infty]$ given by
\begin{align*}
    R(\mu\parallel \nu)=\begin{cases}
        \int_S\log\left(\frac{d\mu}{d\nu}\right)d\mu,\qquad \mu\ll\nu,\\
        +\infty,\qquad\qquad\text{otherwise}.
    \end{cases}
\end{align*}
Here $d\mu/d\nu$ denotes the \textit{Radon-Nikodym} derivative of $\mu$ with respect to $\nu$ (when well-defined).

For a measurable space $(Y, \mathcal{Y})$, let $q(y,dx)$ be a collection of probability measures on $S$ parametrized by $y \in Y$: $q(y, \cdot) \in \mathcal{P(S)}$ for $y \in Y$. Such a $q$ is called a \textit{stochastic kernel} on $S$ given $Y$ if, for every $A \in \mathcal{B} (S)$, $y \mapsto q(y,A) \in [0,1]$ is a measurable function. 
The \textit{transition kernel} of a Markov chain $\{ X_i \} _{i \geq 0}$ taking values in $S$ is a stochastic kernel $q$, such that the conditional distribution of $X_{i+1}$ given $X_{i}$ is $q(X_{i},\cdot)$. The notation $q^{j}(x,\cdot)$ is used for the $j$-th iterate of the transition kernel, i.e.,
\begin{equation*}
    q^{j}(x,A)=\mathbb{P}(X_{i+j}\in A|X_i=x).
\end{equation*}
Given a transition kernel $q$, a set $C\subset S$ is called \textit{small} if there exist $j\in \mathbb{N}$, $\varepsilon>0$ and a probability measure $\nu\in\mathcal{P}(S)$ such that 
\begin{equation}
\label{eq:small}
    q^j(x,A)\ge \varepsilon\nu(A),
\end{equation}
for all $x\in C$ and all $A\in\mathcal{B}(S)$. For a measure $\mu \in \mathcal{P}(S)$ and a transition kernel $q(x, dy)$, we say that $\mu$ is \textit{invariant} for $q$, or for the corresponding Markov chain, if for all $A \in \mathcal{B} (S)$,
\begin{align*}
    \mu (A) = \int _S q(x, A) \mu (dx).
\end{align*}
We say that the Markov chain $\{ X_i \}_i$ with $q$ as transition kernel is \textit{uniformly ergodic} if there exist $R<\infty$ and $r>1$ such that 
\begin{equation*}
    \lVert q^{i}(x_0,\cdot)-\pi\rVert_{TV}\le Rr^{-i}, \ \ \forall x_0 \in S,
\end{equation*}
and \textit{geometrically ergodic} if there exist $R:S\to(0,\infty)$ and $r>1$ such that, for $\pi$-almost every $x_0\in S$, 
\begin{equation*}
     \lVert q^{i}(x_0,\cdot)-\pi\rVert_{TV}\le R(x)r^{-i}.
\end{equation*}
We recall that (see, e.g., \cite{meyn2009markov}) geometric ergodicity is equivalent to the drift and minorization conditon: the existence of a Lyapunov function $V:S\to[1,\infty)$, $\lambda<1$, $b<\infty$ and a small set $C\subset S$ such that 
\begin{equation}
    \label{driftCondition}
        \int_SV(y)q(x,dy)\le \lambda V(x)+bI\{x\in C\}.
 \end{equation}

Lastly, for $\{X_i\} _{i \geq 0}$ a Markov chain on $S$, for each $n \in \mathbb{N}$, the associated \textit{empirical measure} $L^n\in\mathcal{P}(S)$ is defined as
\begin{align}
\label{eq:empMeas}
    L^n (\cdot) = \frac{1}{n} \sum _{i=0} ^{n-1} \delta _{X_i} (\cdot).
\end{align}

\subsection{The Metropolis-Hastings algorithm}
\label{sec:MH}
We now give a brief description of the Metropolis-Hastings (MH) algorithm, introduced in \cite{metropolis1953equation} and \cite{hastings1970monte}, which provides a way to generate a Markov chain $\{X_i\}_{i\ge0}$ on $S$, with a given target distribution $\pi\in\mathcal{P}(S)$ as its invariant distribution. As indicated in Section \ref{sec:notationDef}, in this paper we consider the state spaces $S = \mathbb{R} ^d$ with $d\in\mathbb{N}$; see \cite{Tierney98} for more general settings.

The key ingredient of the MH algorithm is the \textit{proposal distribution} $J(\cdot|x)\in\mathcal{P}(S)$, which is assumed to be defined for all $x\in S$; with a slight abuse of notation we denote by $\pi(\cdot)$ and $J(\cdot|x)$ both the measures and the corresponding probability density functions. 
Here we consider target and proposal distributions that have a density with respect to Lebesgue measure: $\pi \ll \lambda$ and $J(\cdot | x ) \ll \lambda$ for almost all $x \in S$. 

The MH algorithm is as follows: Assume that at step $i$, the chain is in state $x$, $X_i = x$. A proposal $Y_{i+1}$ for the next value of the chain, $X_{i+1}$, is obtained by sampling from the proposal distribution  $J(\cdot|x)$. This proposal is then accepted or rejected according to the \textit{Hastings ratio}, defined as
\begin{equation}
\label{eq:HastingsRatio}
    \varpi(x,y)=\min\left\{1,\frac{\pi(y)J(x|y)}{\pi(x)J(y|x)}\right\}.
\end{equation}
That is, with probability $\varpi(x,Y_{i+1})$,
 we accept the proposal and set $X_{i+1}=Y_{i+1}$. Otherwise, with probability $1-\varpi(x,Y_{i+1})$, we reject the proposal and set $X_{i+1}=x$; for a more detailed description and discussion see \cite{andrieu2003introduction, robert2004monte} and the references therein.

\subsubsection{The Metropolis-Hastings transition kernel.}
\label{sec:MHtransKer}
In analysing both large deviation and ergodicity properties of MH chains, a central object is the associated transition kernels. For a Markov chain generated via the MH algorithm, the transition kernel, which we henceforth denote by $K$, is of the form
\begin{equation}
\label{eq:decompositionK}
    K(x,dy)=a(x,y)dy+r(x)\delta_x(dy),
\end{equation}
where $a(x,y)$ is given by 
\begin{equation}
    \label{eq:acceptanceDensity}
    a(x,y)=\varpi(x,y)J(y|x).
\end{equation}
This term corresponds to moves that are proposed via $J(\cdot|x)$, and accepted with probability $\varpi(x,y)$. The second term on the right-hand side of \eqref{eq:decompositionK} represents transitions to the current state, i.e., no move. This is caused by the proposed state being rejected, which occurs with probability
\begin{equation}
\label{eq:def_r}
    r(x)=1-\int_S a(x,y)dy.
\end{equation}

We henceforth refer to $K$ on the form \eqref{eq:decompositionK}, for some proposal distribution $J$, as a \textit{Metropolis-Hastings kernel}, or MH kernel. An important observation is that, due to the definition of the Hastings ratio, and by extension the MH kernel $K$, under mild assumptions on $J$, the corresponding Markov chain has $\pi$ as its unique invariant distribution (see, e.g., \cite{robert2004monte}). 

\subsection{Large deviation principle for the empirical measures of MH chains}
\label{sec:LDP}
Consider a Markov chain $\{ X_i \} _{i \geq 0}$ with state space $S$. We define the associated sequence $\{ L ^n \} _{n \geq 1} \subset\mathcal{P}(S)$ of empirical measures as in \eqref{eq:empMeas}.
We say that the sequence $\{ L^n \} _{n \geq 1}$ satisfies a \textit{large deviation principle} (LDP) with \textit{speed} $n$ and \textit{rate function} $I : \mathcal{P}(S) \to [0, \infty]$, if $I$ is lower semi-continuous, has compact sub-level sets, and for any measurable set $A \subset \mathcal{P} (S)$,
\begin{align*}
    - \inf _{\mu \in A ^\circ} I(\mu) &\leq \liminf _{n \to \infty} \frac{1}{n} \log \mathbb{P} (L ^n \in A ^\circ) \leq \limsup _{n \to \infty} \frac{1}{n} \log \mathbb{P} (L ^n \in \bar A ) \leq - \inf _{\mu \in \bar A} I (\mu).    
\end{align*}
These inequalities suggest that, as $n\to\infty$,
\begin{align}
\label{eq:asymptoticA}
    \mathbb{P}(L^n\in A)\simeq \exp\left\{-n\,\cdot\inf_{\mu\in A}I(\mu)\right\}.
\end{align}
To give an idea of how this type of asymptotics can be used in the analysis of MCMC methods, consider the following example: let $B_\varepsilon(\pi)$ be the ball of radius $\varepsilon>0$ centered at $\pi$, the invariant measure of the Markov chain, with respect to some metric on the space $\mathcal{P}(S)$ (e.g., the Lévy-Prohorov metric). Then, \eqref{eq:asymptoticA} implies that
\begin{equation}
\label{eq:convergence_prob}
\mathbb{P}\left(L^n\in B_\varepsilon(\pi)\right)=1-\mathbb{P}\left(L^n\in B^\complement_\varepsilon(\pi)\right) \simeq 1-\exp\left\{-n\,\cdot\inf_{\mu\in B^\complement_\varepsilon(\pi)  }I(\mu)\right\},
\end{equation}
that is, the probability that $L^n$ is ``close'' to $\pi$ (corresponding to $L^n\in B_\varepsilon(\pi)$) approaches 1 at an \textit{exponential rate}, and the larger the quantity $\inf_{\mu\in B^\complement_\varepsilon(\pi)  }I(\mu)$, the faster this convergence is. 

Suppose now that $L^n$ is associated with an MH algorithm, with target measure $\pi$, and that there is a corresponding LDP. From \eqref{eq:convergence_prob}, we deduce that  for small $\delta$, we need to run the MH algorithm for $n\gtrsim -\log\delta/\inf_{\mu \in B^\complement_\varepsilon(\pi)}I(\mu)$ in order to get a sample $\{X_i\}_{i=1}^n$ whose empirical distribution satisfies $\mathbb{P}(L^n\in B_\varepsilon(\pi))\gtrsim 1-\delta$; similar ideas, in the context of importance sampling, can be found in \cite{HultNyquist16, CD18}. 
The large deviation principle indicates that MH algorithms whose empirical measures $\{L^n\}$ satisfy an LDP with a larger value of $\inf_{\mu\in B^\complement_\varepsilon(\pi)  }I(\mu)$ are associated with faster convergence in the sense of \eqref{eq:convergence_prob}. This can be used to compare different versions of the MH algorithm with each other, and to tune MH hyperparameters, as described in \cite{milinanni2024b}.

For a thorough treatment of the theory of large deviations and its applications, well beyond the setting of Markov chains and MCMC considered in this paper, see, e.g., \cite{DZ09, FK06, budhiraja2019analysis} and references therein.

In our recent work \cite{milinanni2024large}, we consider Markov chains $\{X_i\}_{i\ge 0}$ generated via the MH algorithm, described in Section \ref{sec:MH}, on a continuous state space that is a subset of $\mathbb{R}^d$. We prove that, under Assumptions \ref{ass:targetAbsContLambda}-\ref{ass:compactSpaceOrLyapunov}, the sequence of associated empirical measures $\{L^n\}_{n\ge 1}$ satisfies an LDP on $\mathcal{P}(S)$ with speed $n$ and rate function $I: \mathcal{P}(S) \to [0, \infty]$ given by
\begin{align}
\label{eq:rateFunc}
    I(\mu)=\inf_{\gamma\in A(\mu)} R(\gamma \parallel \mu \otimes K),
\end{align}
where $A(\mu)$ is defined in \eqref{eq:defA}. The rate function in \eqref{eq:rateFunc} also admits alternative representations that are not based on relative entropy, see \cite{milinanni2024b}.

The following are the assumptions and main result of \cite{milinanni2024large}.
\begin{enumerate}[label=(A.\arabic*), ref=(A.\arabic*)]
\item{\label{ass:targetAbsContLambda}
    $S$ is an open subset of $\mathbb{R}^d$ and the target probability measure $\pi$ is equivalent to $\lambda$ on $S$ (i.e., $\pi\ll\lambda$ and $\lambda\ll\pi$). The probability density $\pi(x)$ is a continuous function.}
\item{\label{ass:proposalDistributionAbsCont}
    The proposal distribution $J(\cdot|x)$ is absolutely continuous with respect to the target measure $\pi$ (i.e., $J(\cdot|x)\ll\pi$), for all $x\in S$. The probability density  $J(y|x)$ is a continuous and bounded function of $x$ and $y$, and satisfies $ J(y|x)>0, \ \forall (x,y)\in S^2.$}

\item{\label{ass:compactSpaceOrLyapunov}
    There exists a function $U:S\to[0,\infty)$ such that the following properties hold:
    \begin{enumerate}[label=(\alph*)]
        \item{\label{ass:notNegativeInf}} $\inf_{x\in S}\left[U(x)-\log\int_Se^{U(y)}K(x,dy)\right]>-\infty$
        \item{\label{ass:relCompact}} For each $M<\infty$, the set $\left\{x\in S \,:\,U(x)-\log\int_Se^{U(y)}K(x,dy)\le M\right\}$ is a relatively compact subset of $S$.
        \item{\label{ass:supU}} For every compact set $K\subset S$ there exists $C_K<\infty$ such that $\sup_{x\in K}U(x)\le C_K$.
    \end{enumerate}}
\end{enumerate}
\begin{theorem}[Theorem 4.1 in \cite{milinanni2024large}]
\label{thm:LDP_SPA}
Let $\{X_i\}_{i\ge 0}$ be the MH chain from Section \ref{sec:MH} and $K(x,dy)$ the associated MH kernel. Let $\{L^n\}_{n\ge 1}\subset\mathcal{P}(S)$ be the corresponding sequence of empirical measures. Under Assumptions \ref{ass:targetAbsContLambda}-\ref{ass:compactSpaceOrLyapunov}, $\{L^n\}_{n\ge 1}$ satisfies an LDP with speed $n$ and rate function given by \eqref{eq:rateFunc}.
\end{theorem}

Assumption~\ref{ass:compactSpaceOrLyapunov} is always satisfied when the state space is bounded—for example, the function $U\equiv 0$ fulfills properties~\ref{ass:notNegativeInf}-\ref{ass:supU}. When the space $S$ is non-compact, e.g., the case $S=\mathbb{R}^d$ considered here, it is often challenging to show whether or not there exists a function $U$ satisfying all three properties of \ref{ass:compactSpaceOrLyapunov}. Note that, in line with \cite{DupuisEllis, budhiraja2019analysis}, we here refer to a function $U$ that satisfies \ref{ass:compactSpaceOrLyapunov} as a Lyapunov function. This differs slightly from the standard Markov chain literature where the term is often reserved for functions $V$ appearing in the drift condition \eqref{driftCondition}.

As outlined in Section \ref{sec:intro}, the aim of this work is to investigate whether or not an LDP holds for some specific choices of proposal distributions $J$. Using Theorem \ref{thm:LDP_SPA}, this amounts to considering \ref{ass:targetAbsContLambda}--\ref{ass:compactSpaceOrLyapunov}, for the different choices of $J$. Before moving to the three choices for proposal distribution--IMH, MALA and RWM--considered in this paper, in Section \ref{sec:propLyapRd} we first obtain an equivalent formulation of Property~\ref{ass:relCompact}, the most challenging part of \ref{ass:compactSpaceOrLyapunov}, more amenable to analysis in the MCMC setting.

\section{An equivalent asymptotic formulation of \texorpdfstring{\ref{ass:compactSpaceOrLyapunov}}{(A.3)}}
\label{sec:propLyapRd}

Establishing whether or not Assumption \ref{ass:compactSpaceOrLyapunov} holds is the most challenging step when using the results of \cite{milinanni2024large} to obtain an LDP for specific choices of proposal distribution $J$. Therefore, before considering specific examples, in this section we inspect the assumption in some more detail. In particular, we obtain an equivalent formulation of the most demanding part, the relative compactness appearing in \ref{ass:relCompact}, that is tailored to the MCMC setting.

Given a function $U:\mathbb{R}^d\to[0,\infty)$, we define $F_U:S\to\mathbb{R}$ as
\begin{equation}
\label{eq:expressionUdefinition}
    F_U(x)=U(x)-\log\int_Se^{U(y)}K(x,dy).
\end{equation} 
Using this definition, we can reformulate Property~\ref{ass:notNegativeInf} of Assumption \ref{ass:compactSpaceOrLyapunov} as $\inf_{x\in S}F_U(x)>-\infty$, and Property~\ref{ass:relCompact} now requires that the sub-level sets of $F_U$ are relatively compact.

To facilitate a comparison with the usual drift condition \eqref{driftCondition}, we reformulate the latter in terms of the function $U=\log V$, where $V$ is the (regular) Lyapunov function. For this $U$, condition \eqref{driftCondition} becomes
     \begin{equation}
     \label{eq:driftCondU}
        U(x)-\log\int_Se^{U(y)}q(x,dy)\ge -\log\left(\lambda + e^{-U(x)}bI\{x\in C\}\right).
    \end{equation}
Note that $F_U$ then corresponds to the left-hand-side of \eqref{eq:driftCondU} for the specific choice of kernel $q=K$. Thus, there is a direct link between the function $F_U$ appearing in the large deviation context and the Lyapunov function $V$ appearing in the drift condition \eqref{driftCondition} associated with the MH kernel $K$. Since geometric ergodicity is equivalent to the drift and minorization conditions (see Section \ref{sec:notationDef} and Chapter 15 in \cite{meyn2009markov}), there is a link between the type of Lyapunov function used in \cite{milinanni2024large} to prove an LDP for MH chains and those used in the Markov chain literature to establish geometric ergodicity. Along those lines, we have the following result concerning Property~\ref{ass:notNegativeInf} of Assumption \ref{ass:compactSpaceOrLyapunov} and the drift condition; geometric ergodicity and large deviations are further discussed in Section \ref{sec:conjecture}.
\begin{proposition}
\label{prop:PropA3a}
    Let $\{X_i\}_{i\ge0}$ be a geometrically ergodic Metropolis-Hasting Markov chain, and let $V:S\to[1,+\infty)$ be a Lyapunov function that satisfies the drift condition~\eqref{driftCondition}. Then Property~\ref{ass:notNegativeInf} in Assumption~\ref{ass:compactSpaceOrLyapunov} holds for the function $U=\log V$.
\end{proposition}
\begin{proof}
   The inequality \eqref{driftCondition} in the standard drift condition for Markov chains is equivalent to \eqref{eq:driftCondU}, i.e., for some $\lambda<1,b<\infty$ and a small set $C\subset S$,
    \begin{equation}
    \label{eq:driftCondFU}
        F_U(x)\ge -\log\left(\lambda + e^{-U(x)}bI\{x\in C\}\right),
    \end{equation}
    with $F_U$ as in \eqref{eq:expressionUdefinition}. 
    Note that $b$ in the drift condition \eqref{driftCondition} can be chosen positive. Therefore, assume $b>0$. For every $x\in S$, because $U(x)\ge0$, we have that $\lambda+be^{-U(x)}I\{x\in C\}\le \lambda+b<\infty$. Combining this with \eqref{eq:driftCondFU}, we obtain $\inf_{x\in S}F_U(x)\ge-\log\left(\lambda+b\right)>-\infty$, which completes the proof.
\end{proof}

As mentioned above, Property~\ref{ass:relCompact} is the more demanding part of \ref{ass:compactSpaceOrLyapunov}. In Lemma \ref{lem:rto0}, we provide necessary and sufficient conditions on the MH kernel $K$ and $U$ that are equivalent to relative compactness of $F_U$ but more amenable to analysis; the proof is straightforward. These conditions are used extensively in Sections \ref{sec:proofIMH}-\ref{sec:proofRWM} and in the \hyperref[appn]{Appendix}, where we consider kernels $K$ corresponding to IMH, MALA and RWM samplers.

\begin{lemma}
\label{lem:rto0}
    For a given function $U:S\to[0,\infty)$, Property~\ref{ass:relCompact} in Assumption~\ref{ass:compactSpaceOrLyapunov} holds if and only if the transition kernel $K$, defined in~\eqref{eq:decompositionK}, satisfies\begin{equation}
    \label{eq:intAto1}
        \lim_{|x|\to\infty}\int_Sa(x,y)dy=1,
    \end{equation}
    and $U$ satisfies 
    \begin{equation}
        \label{eq:intexpUato0}
        \lim_{|x|\to\infty}\int_Se^{U(y)-U(x)}a(x,y)dy=0.
    \end{equation}
    
\end{lemma}
    It is worth emphasising that the function $a$, and therefore also the MH kernel $K$, depends on the specific choice of target $\pi$ and proposal density $J$, and that \eqref{eq:intAto1} is a property solely of $K$. That is, this property does not involve any choice of (potential) Lyapunov function $U$. Therefore, as we will see in the coming sections, for a specific choice of MH dynamics, it is possible to have \eqref{eq:intAto1} satisfied but there being no function $U$ that satisfies \eqref{eq:intexpUato0}. We will also see examples where \eqref{eq:intAto1} is not satisfied and thus there is no reason to look for a suitable Lyapunov function $U$.

\begin{proof}[Proof of Lemma \ref{lem:rto0}]
    From \eqref{eq:expressionUdefinition}, Property~\ref{ass:relCompact} in Assumption~\ref{ass:compactSpaceOrLyapunov} is equivalent to $F_U$ having relatively compact sub-level sets. Because we here consider $S=\mathbb{R}^d$, this in turn holds if and only if for all $M\in \mathbb{R}$, there exists an $R>0$ such that for all $x$ in the sub-level set $\{x\in\mathbb{R}^d:F_U(x)\le M\}$, we have $|x|\le R$. This is equivalent to the following statement: for all $M\in\mathbb{R}$ there exists an $R>0$ such that for all $x$ with norm $|x|>R$, $F_U(x)>M$ holds, i.e.,
    \begin{equation}
    \label{eq:limFU}        \lim_{|x|\to\infty}F_U(x)=+\infty.
    \end{equation}
    
    Using the decomposition \eqref{eq:decompositionK} of the transition kernel $K(x,dy)$, $F_U$ can be rewritten as 
\begin{align}
\label{eq:expressionU}
\begin{split}
    F_U(x)&=-\log\int_Se^{U(y)-U(x)}K(x,dy)=-\log\left(\int_Se^{U(y)-U(x)}a(x,y)dy+r(x)\right).
\end{split}
\end{align}
    From this we obtain that \eqref{eq:limFU} is equivalent to
    \begin{equation}
    \label{eq:limLemma1}
         \lim_{|x|\to\infty}\left(\int_Se^{U(y)-U(x)}a(x,y)dy+r(x)\right)=0.
    \end{equation}
    By definition of the functions $a$ and $r$, we have $a(x,y)\ge0$ and $r(x)\ge0$ for all $x,y\in S$. Therefore, \eqref{eq:limLemma1} is equivalent to satisfying both \eqref{eq:intexpUato0} and $\lim_{|x|\to\infty}r(x)=0$. From the definition~\eqref{eq:def_r} of $r(x)$, the latter limit can be reformulated as \eqref{eq:intAto1}. We conclude that \eqref{eq:intAto1} and \eqref{eq:intexpUato0} are necessary and sufficient conditions for the relative compactness of the sub-level sets of $F_U(x)$, and therefore for Property~\ref{ass:relCompact} in Assumption~\ref{ass:compactSpaceOrLyapunov}.
\end{proof}

In the context of the MH algorithm, the measurable function $r(x)$ represents the probability of rejecting a proposed state when the current state of the chain is $x$, while $1-r(x)$ is the probability of accepting the proposal. Thus, Lemma~\ref{lem:rto0} indicates that $1-r(x)$, the probability of accepting a state proposed from state $x$, converges to $1$ as $|x|\to\infty$.

Equipped with Lemma \ref{lem:rto0}, we are now ready to consider Assumption \ref{ass:compactSpaceOrLyapunov} for the three classes of samplers mentioned in Section \ref{sec:intro}: Independent Metropolis-Hastings (Section \ref{sec:IMH}), the Metropolis-adjusted Langevin algorithm (Section \ref{sec:MALA}), and Random Walk Metropolis (Section \ref{sec:RWM}).

\section{Main results}
\label{sec:main}
In this section we present the main results of the paper, which cover three of the most common MH algorithms. The first choice of MH dynamics we consider is the Independent Metropolis-Hastings (IMH) algorithm \cite{Tie94, robert2004monte}, treated in Section \ref{sec:IMH}; see also Section \ref{sec:intro} for more general references about IMH and its theoretical properties. Next, we consider the Metropolis-adjusted Langevin algorithm (MALA) \cite{besag1994comments, roberts1996exponential} in Section \ref{sec:MALA}. Section \ref{sec:RWM} contains the final example of an explicit MH algorithm considered in this paper: the Random Walk Metropolis (RWM) algorithm; see, e.g., \cite{robert2004monte, mengersen1996rates, roberts2003linking} and references therein. Whereas for IMH and MALA we prove that the LDP of Section \ref{sec:LDP} holds for certain parameters, for the RWM we show that the results of \cite{milinanni2024large} cannot be employed, regardless of the choice of target distribution. Note that all proofs are postponed to Section \ref{sec:proofs} and the \hyperref[appn]{Appendix}.

\subsection{Independent Metropolis-Hastings algorithm}
\label{sec:IMH}
 In the IMH algorithm, moves from a state $x\in S$ are proposed with a proposal distribution $J(\cdot|x)$ that is independent of $x$, i.e. $J(dy|x)=\hat J(dy)\in\mathcal{P}(S)$. Let $\hat J(y)$ be the probability density of $\hat J(dy)$. With this choice of proposal distribution, the density $a(x,y)$ of the acceptance part in the MH kernel $K(x,dy)$ simplifies to
\begin{equation*}
    a(x,y)=\min\left\{1,\frac{\pi(y)\hat J(x)}{\pi(x)\hat J(y)}\right\}\hat J(y).
\end{equation*}

We consider target and proposal distributions with densities of the form
\begin{equation*}
    \pi(x)\propto e^{-\eta|x|^\alpha}\quad\text{and}\quad \hat J(y)\propto e^{-\gamma|y|^\beta},
\end{equation*}
respectively, with $\eta,\gamma,\alpha,\beta>0$.
In Theorem \ref{thm:LDP_indep} we prove an LDP for the IMH sampler for certain values of these hyperparameters.

We consider a target $\pi$ and proposal $\hat J$ on these forms in order to facilitate comparison with the standard results of \cite{mengersen1996rates}. This comparison is carried out in Section~\ref{sec:conjecture}.

\begin{theorem}
    \label{thm:LDP_indep}
    Consider the target density $\pi(x)\propto e^{-\eta|x|^\alpha}$ and the independent proposal density $\hat J(y)\propto e^{-\gamma|y|^\beta}$ in the Independent Metropolis-Hastings algorithm. Suppose that either of the following holds:
    \begin{enumerate}
        \item[i)] $\alpha=\beta$ and $\eta\ge\gamma$,
        \item[ii)] $\alpha > \beta$.
    \end{enumerate}
    Then, the empirical measures of the associated Metropolis-Hastings chain satisfies an LDP with speed $n$ and rate function $I$ given by  \eqref{eq:rateFunc}.
\end{theorem}

The LDP in Theorem \ref{thm:LDP_indep} follows directly from the combination of Theorem \ref{thm:LDP_SPA}, the general large deviation result for empirical measures of MH chains, and the following result on the existence of a suitable Lyapunov function, the proof of which is carried out in section~\ref{sec:proofIMH}.

\begin{proposition}
    \label{prop:LyapunovForIMH}
    Consider the target density $\pi(x)\propto e^{-\eta|x|^\alpha}$ and the independent proposal density $\hat J(y)\propto e^{-\gamma|y|^\beta}$ in the Independent Metropolis-Hastings algorithm. Assumption~\ref{ass:compactSpaceOrLyapunov} is satisfied
    if and only if either of the following holds:
    \begin{enumerate}
        \item[i)] $\alpha=\beta$ and $\eta\ge\gamma$,
        \item[ii)] $\alpha > \beta$.
    \end{enumerate}
\end{proposition}

\subsection{Metropolis-adjusted Langevin algorithm}
\label{sec:MALA}
MALA is characterised by the following proposal density,
\begin{equation*}
    J(y|x) = C\exp\left\{-\frac{1}{2\varepsilon} \left|y-x-\frac{\varepsilon}{2}\nabla \log\pi(x)\right|^2\right\}.
\end{equation*}
This proposal is obtained by a discretisation with step size $\varepsilon>0$ of the 
continuous-time Langevin process $X = \{ X_t \} _{t \geq 0}$ in  $\mathbb{R}^d$, defined by
\begin{equation*}
    dX_t=\frac{1}{2}\nabla \log\pi(X_t)dt+dB_t,
\end{equation*}
where $B_t$ denotes the standard $d$-dimensional Brownian motion.

The aim of the section is to provide necessary and sufficient conditions for the existence of a Lyapunov function $U(x)$ satisfying Assumption~\ref{ass:compactSpaceOrLyapunov} when the target density is
\begin{equation}
\label{eq:MALAtarget}
    \pi(x)\propto e^{-\gamma|x|^\beta},
\end{equation}
with $\beta>0$. In this case the corresponding MALA proposal density is given by
\begin{equation}
\label{eq:MALAproposal}
            J(y|x)=C\exp\left\{-\frac{1}{2\varepsilon}\left| y-x+\frac{\varepsilon\gamma\beta}{2}|x|^{\beta-2}x\right|^2\right\}.
\end{equation}
Such necessary and sufficient conditions are obtained in Proposition \ref{prop:MALA}. 
Similar to the results for IMH in Section \ref{sec:IMH}, combined with the large deviation result in \cite{milinanni2024large}, Proposition \ref{prop:MALA} yields the following LDP for the empirical measure of the MH chain with MALA proposal.
\begin{theorem}
\label{thm:LDP_MALA}
    Consider a target density $\pi(x)\propto e^{-\gamma|x|^\beta}$ and let $J(y|x)$ be the corresponding MALA proposal density with discretization step $\varepsilon$,
    \begin{equation*}
        J(y|x) \propto \exp \left\{ -\frac{1}{2\varepsilon} \left|y - x + \frac{\varepsilon \gamma \beta}{2} |x| ^{\beta - 2} x\right| ^2 \right\}.
    \end{equation*}
    Suppose that either of the following holds:
    \begin{itemize}
    \item[i)] $\beta=2$ and $\varepsilon\gamma<2$, \item[ii)] $1<\beta<2$. 
    \end{itemize}
    Then, the empirical measures of the associated Metropolis-Hastings chain satisfy an LDP with speed $n$ and rate function $I$ given by \eqref{eq:rateFunc}.
\end{theorem}

The proof of Theorem \ref{thm:LDP_MALA} is an immediate consequence of combining Theorem \ref{thm:LDP_SPA} with the following result, the proof of which is outlined in Section \ref{sec:proofMALA} and full details provided in the \hyperref[appn]{Appendix}.
\begin{proposition}
\label{prop:MALA}
    Let the target density be $\pi(x)\propto e^{-\gamma|x|^\beta}$ and let $J(y|x)$ be the corresponding MALA proposal density \eqref{eq:MALAproposal} with discretisation step $\varepsilon$. Assumption~\ref{ass:compactSpaceOrLyapunov} is satisfied if and only if either of the following holds:
    \begin{enumerate}
        \item[i)] $\beta=2$ and $\varepsilon\gamma<2$,
        \item[ii)] $1<\beta<2$. 
    \end{enumerate}
\end{proposition}

\subsection{Random Walk Metropolis algorithm}
\label{sec:RWM}
The RWM is characterized by proposal densities $J(y|x)$ of the form 
\begin{equation*}
    J(y|x)=\hat J(y-x)=\hat J(x-y),
\end{equation*}
where, with an abuse of notation, $\hat J(t)=\hat J(-t)$ is the density of some probability distribution in $\hat J(\cdot)\in\mathcal{P}(S)$. The proposal density is therefore symmetric, i.e. $J(y|x)=J(x|y)$, and the Hastings ratio \eqref{eq:HastingsRatio} simplifies to
\begin{equation*}
    \varpi(x,y)=\min\left\{1,\frac{\pi(y)}{\pi(x)}\right\}.
\end{equation*}
With the following proposition we show that \ref{ass:compactSpaceOrLyapunov} cannot hold when employing the RWM proposal.

\begin{proposition}
\label{prop:RWMnoLyapunov}
    Let 
    \begin{equation}
        \label{eq:RWMJ}
        J(y|x)=\hat J(y-x) = \hat J(x-y)
    \end{equation}
    be the proposal density in the RWM. There exists no function $U:S \to[0,\infty)$ that satisfies \ref{ass:compactSpaceOrLyapunov}.
\end{proposition}
The proof of Proposition~\ref{prop:RWMnoLyapunov} is provided in Section~\ref{sec:proofRWM}.

\section{LDP for MH chains, Lyapunov condition and geometric ergodicity: Summary and discussion}
\label{sec:conjecture}
In Section \ref{sec:main}, we consider different instances of the Metropolis-Hastings algorithm and show under what conditions Assumption~\ref{ass:compactSpaceOrLyapunov} is satisfied. In the first two examples--IMH and MALA--we find a Lyapunov function that satisfies \ref{ass:compactSpaceOrLyapunov} for certain values of the algorithms' hyperparameters. However, for RWM we instead prove that a Lyapunov function as in~\ref{ass:compactSpaceOrLyapunov} cannot exist for any combination of proposal and target distributions. 

  Assumption~\ref{ass:compactSpaceOrLyapunov}, in conjunction with \ref{ass:targetAbsContLambda} and \ref{ass:proposalDistributionAbsCont}, guarantees that the empirical measure of the algorithm's Markov chain satisfies an LDP with speed and rate function described in Section~\ref{sec:LDP}. This combined with the results of Sections \ref{sec:IMH}-\ref{sec:MALA} allow us to state Theorems~\ref{thm:LDP_indep} and \ref{thm:LDP_MALA}, providing LDPs for the empirical measures of IMH and MALA chains, respectively, under certain conditions on their parameters. For MALA, we believe that Proposition \ref{prop:MALA} and Theorem \ref{thm:LDP_MALA} can be extended to target distributions $\pi\in\mathcal{P}(\mathbb{R}^d)$ of the form
    \begin{align}
    \label{eq:pi_exp_gamma_x_beta}
        \pi(x)\propto e^{-\gamma| x|^\beta},\qquad |x|\ge R,
    \end{align}
    for some $R>0$, that is, the restriction is only on the tail decay. For the RWM, we emphasise that the fact that Assumption~\ref{ass:compactSpaceOrLyapunov} can never be satisfied does not imply that an empirical measure LDP cannot hold for MH chains associated with RWM dynamics. Rather, we believe that an LDP should exist for certain choices of proposal and target, and that the three assumptions \ref{ass:targetAbsContLambda}-\ref{ass:compactSpaceOrLyapunov} in this case are only sufficient, not necessary, for an LDP. Should such an LDP hold, the associated rate function does not have to agree with that of Theorem \ref{thm:LDP_SPA}. Because of some similarities with reflected Brownian motion and related constrained processes, the techniques and results of \cite{BD03} may be well-suited to treat the RWM case, as well as the ``boundary case'' $\beta = 1$ in MALA (see Table \ref{tab:GE} and the discussion below).

We now compare our results to existing results on geometric ergodicity for the MH chains corresponding to the three classes of MCMC samplers. The conclusions are summarised in Table \ref{tab:GE}.

\begin{table}[b]
  \centering
  \begin{tabular}{cccc}
    \multicolumn{2}{c}{} & \shortstack{Assumption \\ \ref{ass:compactSpaceOrLyapunov}} & \shortstack{Geometric\\ergodicity}\\
    \toprule
    \multirow{2}{*}{IMH} & $\alpha=\beta$ and $\eta>\gamma$,
         or $\alpha > \beta$& \cellcolor{light green}\ding{51} & \cellcolor{light green}\ding{51}  \\ 
    &otherwise & \cellcolor{brick!50}\ding{55} & \cellcolor{brick!50}\ding{55} \\
    \hline
    \multirow{3}{*}{\makecell{MALA\\$d=1$}} &  $\beta=2$ and $\varepsilon\gamma<2$, or $\beta\in(1,2)$& \cellcolor{light green}\ding{51}  & \cellcolor{light green}\ding{51}  \\
    & $\beta=1$ & \cellcolor{brick!50}\ding{55}  & \cellcolor{light green}\ding{51}$^*$  \\
    & $\beta=2$ and $\varepsilon\gamma\ge2$, or $\beta\in(0,1)\cup(2,+\infty)$&\cellcolor{brick!50}\ding{55}&\cellcolor{brick!50}\ding{55}\\
    \hline

    \multirow{3}{*}{\makecell{MALA\\$d\ge1$}} &  $\beta=2$ and $\varepsilon\gamma<2$, or $\beta\in(1,2)$& \cellcolor{light green}\ding{51}  & \cellcolor{light green}\ding{51}  \\
    & $\beta=1,  \sqrt{\varepsilon}\gamma\gg0$ & \cellcolor{brick!50}\ding{55}& \cellcolor{light green}\ding{51}  \\
    &$\beta=2$ and $\varepsilon\gamma>2$, or $\beta\in(2,+\infty)$&\cellcolor{brick!50}\ding{55}&\cellcolor{brick!50}\ding{55}\\
    \hline
    \multirow{2}{*}{RWM} & tail decays as in \cite{mengersen1996rates} 
    &\cellcolor{brick!50}\ding{55} & \cellcolor{light green}\ding{51} \\
    & otherwise &\cellcolor{brick!50}\ding{55} & \cellcolor{brick!50}\ding{55} \\
    \bottomrule
    \multicolumn{4}{p{0.65\linewidth}}{\footnotesize $^*$The result on geometric ergodicity of MALA with $\beta=1$ in dimension $d=1$ refers to~\cite{roberts1996exponential}, where the state space considered is $S=(0,+\infty)$.}
  \end{tabular}
  \caption{\label{tab:GE}\footnotesize{Summary of the results from Sections~\ref{sec:IMH}-\ref{sec:RWM}, including existing results on geometric ergodicity. For IMH, the target and proposal are taken to be on the forms $\pi(x)\propto \exp\{-\eta|x|^\alpha\}$ and $f(x)\propto \exp\{-\gamma|\eta|^\beta\}$, respectively. For MALA, the target is on the form $\pi(x)\propto -\gamma|x|^\beta$, and the MALA proposal becomes \eqref{eq:MALAproposal}; here the results are split into two cases, corresponding to the different results on geometric ergodicity for $d=1$, considered in \cite{roberts1996exponential} (second row), and the more general case analysed in \cite{RZ23} (third row). Note that the results from~\cite{roberts1996exponential} on geometric ergodicity for MALA with $\beta=1$ in dimension $d=1$ holds on the state space $(0,+\infty)$. For RWM the results refer to any proposal of the form \eqref{eq:RWMJ}.}} 
  \end{table}

For IMH, in \cite{mengersen1996rates} Mengersen and Tweedie show that the associated Markov chain is uniformly ergodic, and hence geometrically ergodic, if the proposal density is bounded below by a multiple of the target; otherwise even geometric ergodicity fails. In Proposition~\ref{prop:LyapunovForIMH}, we characterise the conditions on the parameters $\eta,\gamma,\alpha,\beta$ that are necessary and sufficient for the existence of a Lyapunov function $U$ satisfying Property~\ref{ass:relCompact} of \ref{ass:compactSpaceOrLyapunov}. It turns out that these conditions correspond to the cases where the target density $\pi$ has lighter tails than the proposal density $\hat J$, i.e., precisely the cases where the IMH chain is uniformly ergodic.

For MALA, we compare the results of Proposition~\ref{prop:MALA} to the analysis in \cite{roberts1996exponential}. Therein, Roberts and Tweedie analyse one-dimensional target distributions $\pi\in\mathcal{P}(\mathbb{R})$ of the form \eqref{eq:pi_exp_gamma_x_beta}. Thus, in the case of $d=1$, the tail behaviour is the same as in Proposition~\ref{prop:MALA}. The result of their analysis states that the MALA Markov chain is geometrically ergodic when $1<\beta<2$, and when $\beta=2$ and $\varepsilon\gamma<2$. These are precisely the cases in Proposition~\ref{prop:MALA} for which a Lyapunov function exists, and therefore the LDP in Theorem~\ref{thm:LDP_MALA} holds. In \cite{roberts1996exponential} it is also shown that the Markov chain associated with MALA, with the given forms of $\pi$ and $J$, is not geometrically ergodic when $\beta\in(0,1)$, $\beta>2$ or $\beta=2$ and $\varepsilon\gamma\ge2$. For the same values of parameters $\beta,\gamma$ and $\varepsilon$, Proposition~\ref{prop:MALA} states that there does not exist a function $U$ satisfying Assumption~\ref{ass:compactSpaceOrLyapunov}. For the remaining case, $\beta=1$, Roberts and Tweedie cite an argument from \cite{meyn2009markov} showing that the resulting Markov chain is geometrically ergodic for positive $x$. However, our result states that when $\beta=1$ there cannot exist a function $U$ satisfying \ref{ass:compactSpaceOrLyapunov}. 

Results similar to those of \cite{roberts1996exponential} can be obtained also for arbitrary dimensions $d\in\mathbb{N}$ by applying the results from \cite{RZ23}, where Roy and Zhang provide sufficient conditions for geometric ergodicity of Markov chains arising from MALA. If we apply their results (Theorem 1 in \cite{RZ23}) to the setting of Proposition~\ref{prop:MALA}, we obtain that if the conditions \textit{i)} and \textit{ii)} are satisfied, then the MALA chain is geometrically ergodic. In addition, by the same theorem, when $\beta=1$, if the product $\sqrt{\varepsilon}\gamma$ is sufficiently large, the sufficient conditions for geometric ergodicity hold. Consistent with the one-dimensional case, this is an example where the MALA chain is geometrically ergodic, but Assumption~\ref{ass:compactSpaceOrLyapunov} is not satisfied.  Moreover, in the same paper necessary conditions for geometric ergodicity are also derived. Using such results (Theorem 4 in \cite{RZ23}), we establish that if $\beta>2$, or $\beta = 2$ and $\varepsilon\gamma>2$ (cases where Assumption~\ref{ass:compactSpaceOrLyapunov} does not hold), the MALA chain is not geometrically ergodic; we omit the details.

For RWM, Proposition \ref{prop:RWMnoLyapunov} can be related to Theorem 3.1 in \cite{mengersen1996rates}, which states that the Markov chain generated via the RWM algorithm with a proposal distribution of the form under consideration is not uniformly ergodic for any $\pi$. However, in the same paper, Mengersen and Tweedie show that imposing additional assumptions on the target leads to the associated MH chain being geometrically ergodic. In this case, we suspect that if the chain is geometrically ergodic, then an LDP does hold. This is not in conflict with the result of Proposition~\ref{prop:RWMnoLyapunov}, according to which Assumption~\ref{ass:compactSpaceOrLyapunov} is not satisfied for \textit{any} $U$, even when specific tail decays are imposed. In fact, as discussed above, Assumptions~\ref{ass:targetAbsContLambda}-\ref{ass:compactSpaceOrLyapunov} are sufficient for an LDP, and we believe that they are not necessary.

Table \ref{tab:GE} shows that for IMH and the bulk of MALA samplers, i.e., excluding the case $\beta = 1$, conditions for geometric ergodicity coincide precisely with those guaranteeing Assumption \ref{ass:compactSpaceOrLyapunov}, and by extension the cases where we have an affirmative answer to whether or not an LDP holds for the underlying empirical measures. Moreover, as the discussion above hints at, there are reasons to believe that an LDP will hold also in the cases where Assumption \ref{ass:compactSpaceOrLyapunov} is not satisfied but the MH chain is geometrically ergodic. Specifically, Property~\ref{ass:relCompact} is what makes Assumption \ref{ass:compactSpaceOrLyapunov} too strict for some MH chains. To highlight this, consider the case where geometric ergodicity, i.e., the drift and minorization conditions, holds. Proposition \ref{prop:PropA3a} shows that Property~\ref{ass:notNegativeInf} also holds in this setting. Moreover, Property~\ref{ass:supU} is a reasonable assumption on $U$, as, e.g., continuity would be enough, and this holds for the standard functions used to satisfy the drift condition. The missing part is thus Property~\ref{ass:relCompact}.

\section{Proofs of main results}
\label{sec:proofs}
\subsection{Proof for the Independent Metropolis-Hastings algorithm}
\label{sec:proofIMH}
\begin{proof}[Proof of Proposition \ref{prop:LyapunovForIMH}]
    Start by considering Property~\ref{ass:relCompact} in Assumption~\ref{ass:compactSpaceOrLyapunov}. By Lemma~\ref{lem:rto0}, this property is equivalent to satisfying both \eqref{eq:intAto1} and \eqref{eq:intexpUato0}. In the IMH case, limit \eqref{eq:intAto1} can be rewritten as
    \begin{equation}
    \label{eq:IMHrto0}
        \lim_{|x|\to\infty}\int_S\min\left\{1,\frac{\pi(y)\hat J(x)}{\pi(x)\hat J(y)}\right\}\hat J(y)dy=1.
    \end{equation}
    With our choice of $\pi$ and $\hat J$ we have the following pointwise convergence:
      \begin{align*}
        \lim_{|x|\to\infty}\frac{\pi(y)\hat J(x)}{\pi(x)\hat J(y)} &=  e^{-\eta|y|^\alpha+\gamma|y|^\beta}\lim_{|x|\to\infty}e^{\eta|x|^\alpha-\gamma|x|^\beta}=\begin{cases}
            0\quad\quad \text{if}\;\alpha<\beta, \text{ or } \alpha=\beta\;\text{and}\;\eta<\gamma,\\
            1\quad\quad \text{if}\;\alpha=\beta\;\text{and}\;\eta=\gamma,\\
            +\infty\quad \text{if}\;\alpha>\beta,\text{ or }\alpha=\beta\;\text{and}\;\eta>\gamma,\\
        \end{cases}
    \end{align*}
    and therefore
      \begin{align*}
        \lim_{|x|\to\infty}\min\left\{1,\frac{\pi(y)\hat J(x)}{\pi(x)\hat J(y)}\right\} =  \begin{cases}
            0\quad\quad \text{if}\;\alpha<\beta, \text{ or } \alpha=\beta\;\text{and}\;\eta<\gamma,\\
            1\quad\quad \text{if}\;\alpha=\beta\;\text{and}\;\eta\ge\gamma, \text{ or } \alpha >\beta.
        \end{cases}
    \end{align*}
    Because $\hat J$ is, by definition, a probability density, by dominated convergence the limit in \eqref{eq:IMHrto0}, and therefore the first condition \eqref{eq:intAto1} in Lemma~\ref{lem:rto0}, is satisfied if and only if
    \begin{equation*}
        \lim_{|x|\to\infty}\min\left\{1,\frac{\pi(y)\hat J(x)}{\pi(x)\hat J(y)}\right\}=1,
    \end{equation*}
   i.e. if and only if $\alpha=\beta$ and $\eta\ge\gamma$, or $\alpha > \beta$.

    We will now show that by choosing $U(x)=\frac{\gamma}{2} |x|^\beta$, the limit \eqref{eq:intexpUato0} in Lemma~\ref{lem:rto0} is satisfied for any choice of $\eta,\gamma,\alpha,\beta>0$. Therefore, using the above results related to verifying the limit \eqref{eq:intAto1}, we have that Property~\ref{ass:relCompact} in Assumption~\ref{ass:compactSpaceOrLyapunov} is satisfied if and only if $\alpha=\beta$ and $\eta\ge\gamma$, or $\alpha > \beta$.

    Let $C\in\mathbb{R}$ such that $\hat J(y)=Ce^{-\gamma|x|^\beta}$. If $U(x)=\frac{\gamma}{2} |x|^\beta$, the limit in \eqref{eq:intexpUato0} becomes 
    \begin{align*}
        \lim_{|x|\to\infty}\int_Se^{\frac{\gamma}{2} |y|^\beta-\frac{\gamma}{2} |x|^\beta}\min\left\{1,\frac{\pi(y)\hat J(x)}{\pi(x)\hat J(y)}\right\}\hat J(y)dy&\le \lim_{|x|\to\infty}\int_Se^{\frac{\gamma}{2} |y|^\beta-\frac{\gamma}{2} |x|^\beta}\cdot1\cdot Ce^{-\gamma|y|^\beta}dy\\
        &=C\int_Se^{-\frac{\gamma}{2} |y|^\beta}dy\lim_{|x|\to\infty}e^{-\frac{\gamma}{2} |x|^\beta}=0.
    \end{align*}
    Thus, $U(x)=\frac{\gamma}{2} |x|^\beta$ satisfies Property~\ref{ass:relCompact} in Assumption~\ref{ass:compactSpaceOrLyapunov}.

    We proceed by showing that this choice of $U(x)$ also satisfies Property~\ref{ass:notNegativeInf}. Because $a(x,y)\le \hat J(y)=Ce^{-\gamma|x|^\beta}$ and $r(x)\le 1$,
    \begin{align*}
        \inf_{x\in S}&F_U(x)=\inf_{x\in S}\left\{-\log\left(\int_Se^{\frac{\gamma}{2}|y|^\beta-\frac{\gamma}{2}|x|^\beta}a(x,y)dy+r(x)\right)\right\}\\
        &\ge-\sup_{x\in S}\left\{\log\left(\int_Se^{\frac{\gamma}{2}|y|^\beta-\frac{\gamma}{2}|x|^\beta}\cdot Ce^{-\gamma|y|^\beta}dy+1\right)\right\}=-\log\left(C\int_Se^{-\frac{\gamma}{2}|y|^\beta}dy\cdot\sup_{x\in S}e^{-\frac{\gamma}{2}|x|^\beta}+1\right)\\
        &=-\log\left(C\int_Se^{-\frac{\gamma}{2}|y|^\beta}dy\cdot 1+1\right)>-\infty.
    \end{align*}
 Thus, $U(x)=\frac{\gamma}{2}|x|^\beta$ satisfies Property~\ref{ass:notNegativeInf} in Assumption~\ref{ass:compactSpaceOrLyapunov}.

    Lastly, since $U(x)=\frac{\gamma}{2}|x|^\beta$ is continuous, it is bounded on every compact set, hence it also satisfies Property~\ref{ass:supU} in~\ref{ass:compactSpaceOrLyapunov}. This completes the proof.
 \end{proof}

\subsection{Proof for the Metropolis-adjusted Langevin algorithm}
\label{sec:proofMALA}

\begin{proof}[Outline of proof for Proposition 
\ref{prop:MALA}]
The proof of Proposition \ref{prop:MALA} relies on a series of Lemmas, presented in the \hyperref[appn]{Appendix}, where we consider separately different ranges of the parameters $\varepsilon,\gamma$ and $\beta$.

We start by deriving, in Lemma~\ref{lem:equivalentForIntA}, a condition equivalent to the necessary condition \eqref{eq:intAto1} in Lemma~\ref{lem:rto0}, formulated specifically for MALA. Next, we analyze conditions \eqref{eq:intAto1} and \eqref{eq:intexpUato0}, the former via Lemma \ref{lem:equivalentForIntA}, for different ranges of the parameters $\beta, \gamma$ and $\varepsilon$. In Lemmas \ref{lem:beta0to1}-\ref{lem:beta1} we show that for $\beta \in (0,1]$, condition \eqref{eq:intAto1} holds but there does not exist a function $U$ satisfying \eqref{eq:intexpUato0}. Next, in Lemmas \ref{lem:beta1to2}-\ref{lem:beta2} we consider the necessary condition \eqref{eq:intAto1} for $\beta \in (1,2)$ and $\beta =2$, respectively. In Lemma \ref{lem:Uforbeta1to2}, for $\beta \in (1,2]$ and $\gamma, \varepsilon$ such that \eqref{eq:intAto1} holds, we construct a function $U$ satisfying \ref{ass:compactSpaceOrLyapunov}. Lastly, in Lemma \ref{lem:beta>2}, we show that for $\beta >2$, even the condition \eqref{eq:intAto1} related to the MH kernel is violated. 

Proposition \ref{prop:MALA} is an immediate consequence of the combination of these lemmas. In fact, Lemmas~\ref{lem:beta1to2}, \ref{lem:beta2} and \ref{lem:Uforbeta1to2} show that Assumption~\ref{ass:compactSpaceOrLyapunov} is satisfied if $\beta=2$ and $\varepsilon\gamma<2$, or $1<\beta<2$. In all other cases, at least one of the two conditions \eqref{eq:intAto1} or \eqref{eq:intexpUato0} is not satisfied, as shown in Lemmas~\ref{lem:beta0to1}, \ref{lem:beta1}, \ref{lem:beta2} and \ref{lem:beta>2}. It follows by Lemma~\ref{lem:rto0} that there cannot exists a $U$ that satisfies Property~\ref{ass:relCompact}, and therefore Assumption~\ref{ass:compactSpaceOrLyapunov} does not hold for such choices of $\beta,\varepsilon,\gamma$.
\end{proof}

\subsection{Proof for the Random Walk Metropolis algorithm}
\label{sec:proofRWM}
\begin{proof}[Proof of Proposition \ref{prop:RWMnoLyapunov}]
When the proposal density $J(y|x)$ is of random walk type \eqref{eq:RWMJ}, the probability of accepting any proposal from state $x\in S$ can be written as
\begin{equation*}
    \int_S a(x,y)dy = \int_S\varpi(x,y)\hat J(y-x)dy=\int_S\varpi(x,x+t)\hat J(t)dt,
\end{equation*}
where we applied the change of variable $y\mapsto x+t$. Therefore, in the RWM case, the two necessary conditions \eqref{eq:intAto1} and \eqref{eq:intexpUato0} for Property~\ref{ass:relCompact} of~\ref{ass:compactSpaceOrLyapunov} (see Lemma~\ref{lem:rto0})  can be expressed as
    \begin{align}
    \label{eq:RWMlimrto0}
        \begin{split}
        \lim_{|x|\to\infty}\int_S\varpi(x,x+t)\hat J(t)dt=1,
        \end{split}
    \end{align}
and,
    \begin{align}
        \lim_{|x|\to\infty}e^{-U(x)}\int_Se^{U(x+t)}\varpi(x,x+t)\hat J(t)dt=0,\label{eq:RWMintegral}
    \end{align}
    respectively.

    To show that Assumption~\ref{ass:compactSpaceOrLyapunov} cannot be satisfied with the RWM kernel, we assume the necessary condition \eqref{eq:RWMlimrto0} holds, and prove that there cannot then also exist a function $U$ satisfying~\eqref{eq:RWMintegral}.

    Assume, by contradiction, that there exists a function $U:S\to[0,\infty)$ satisfying \eqref{eq:RWMintegral} and construct a sequence $\{U_n\}_{n\in\mathbb{N}}\subset [0,\infty)$ with terms $U_n = \inf_{|y|\ge n}U(y)$. For a fixed $\varepsilon>0$, for each $n\in\mathbb{N}$, there exists $x_n\in S$ such that $|x_n|\ge n$ and
    \begin{equation}
    \label{eq:RWMinequalityU}
        U_n\le U(x_n)< U_n+\varepsilon.
    \end{equation}
    Let $\{x_n\}\subset S$ be such a sequence, i.e., $|x_n| \geq n$ and \eqref{eq:RWMinequalityU} holds for $x_n$. By construction, $\lim_{n\to\infty}|x_n|=+\infty$, and if \eqref{eq:RWMintegral} holds then it also follows that
    \begin{equation}
    \label{eq:RWMintegralSubseq}
        \lim_{n\to\infty}e^{-U(x_n)}\int_Se^{U(x_n+t)}\varpi(x_n,x_n+t)\hat J(t)dt=0.
    \end{equation}
    For each $n$, let $H_n=\{t\, : \, \langle x_n,t\rangle \ge 0\}\subset S$. Observe that for $t\in H_n$, 
    \begin{equation*}
        |x_n+t|=\sqrt{|x_n|^2+|t|^2+2\langle x_n,t\rangle}\ge |x_n| \ge n,
    \end{equation*}
    hence,
    \begin{equation}
    \label{RWM:ineqUonHn}
        U(x_n+t)\ge \inf_{|y|\ge n}U(y)=U_n.
    \end{equation}
    Moreover, $S\setminus H_n = - H_n^{\circ}$ for all $n$, and since $\hat J$ is a symmetric measure (i.e. $\hat J(B)=\hat J(-B)$ for all measurable sets $B$) and $\hat J\ll \lambda$, we have that $\hat J(H_n) = \hat J(H_n^\circ) = \hat J(- H_n^\circ)=\hat J (S\setminus H_n)$. Combining this with $\hat J (H_n)+\hat J(S\setminus H_n)= \hat J(S) = 1$, we obtain $\int_{H_n}\hat J(t)dt=\frac{1}{2}$, which we use to determine
    \begin{equation*}
        \lim_{n\to\infty}\int_{H_n}\varpi(x_n,x_n+t)\hat J(t)dt.
    \end{equation*}
    Because $\varpi(x,y)\le 1 $ for all $x,y\in S$, for any $n \in \mathbb{N}$ we have
    \begin{equation}
    \label{eq:UB_RWM}
        \int_{H_n}\varpi(x_n,x_n+t)\hat J(t)dt \le\int_{H_n}\hat J(t)dt=\frac{1}{2}.
    \end{equation}
    In the other direction, under the assumption that \eqref{eq:RWMlimrto0} holds, we have the asymptotic lower bound
    \begin{equation}
    \label{eq:LB_RWM}
    \begin{split}
\liminf_{n\to\infty}\int_{H_n}&\varpi(x_n,x_n+t)\hat J(t)dt = \liminf_{n\to\infty}\left(\int_{S}\varpi(x_n,x_n+t)\hat J(t)dt-\int_{S\setminus H_n}\varpi(x_n,x_n+t)\hat J(t)dt\right)\\
&\ge \liminf_{n\to\infty}\int_{S}\varpi(x_n,x_n+t)\hat J(t)dt+\liminf_{n\to\infty}\left(-
\int_{S\setminus H_n}\varpi(x_n,x_n+t)\hat J(t)dt\right)\\
&=1-\limsup_{n\to\infty}\int_{S\setminus H_n}\varpi(x_n,x_n+t)\hat J(t)dt\ge1-\limsup_{n\to\infty}\int_{S\setminus H_n}\hat J(t)dt=1-\frac{1}{2}=\frac{1}{2}.
\end{split}
    \end{equation}
    Combining the upper and lower bounds \eqref{eq:UB_RWM} and \eqref{eq:LB_RWM},  if \eqref{eq:RWMlimrto0} holds we also have
    \begin{equation}
    \label{eq:RWMvarpiOnHn}
       \lim_{n\to\infty}\int_{H_n}\varpi(x_n,x_n+t)\hat J(t)dt = \frac{1}{2}. 
    \end{equation}

Using \eqref{eq:RWMinequalityU}, \eqref{RWM:ineqUonHn} and \eqref{eq:RWMvarpiOnHn} we can bound the term on the left-hand side of \eqref{eq:RWMintegralSubseq} from below as follows:
\begin{align*}
        \lim_{n\to\infty}e^{-U(x_n)}\int_Se^{U(x_n+t)}\varpi(x_n,x_n+t)\hat J(t)dt  \ge \lim_{n\to\infty}e^{-U(x_n)}\int_{H_n}e^{U(x_n+t)}\varpi(x_n,x_n+t)\hat J(t)dt\\
        \ge \lim_{n\to\infty}e^{-U_n-\varepsilon}\int_{H_n}e^{U_n}\varpi(x_n,x_n+t)\hat J(t)dt = e^{-\varepsilon}\lim_{n\to\infty}\int_{H_n}\varpi(x_n,x_n+t)\hat J(t)dt=\frac{e^{-\varepsilon}}{2}>0.
    \end{align*}
This contradicts the necessary condition \eqref{eq:RWMintegral}, which requires the limit to be equal to 0. Therefore, \eqref{eq:RWMlimrto0} and \eqref{eq:RWMintegral} cannot both hold in this setting. Using Lemma \ref{lem:rto0} we conclude that a Lyapunov function $U$ satisfying Assumption~\ref{ass:compactSpaceOrLyapunov} cannot exist when the proposal distribution is of RWM type.   
\end{proof}


\begin{appendix}
\label{appn}
\section*{Appendix: Proofs for the Metropolis-adjusted Langevin Algorithm}
In this section, we establish Lemmas~\ref{lem:equivalentForIntA}-\ref{lem:beta>2}, which are used in the proof of Proposition~\ref{prop:MALA} to study the assumptions of the large deviation principle for MALA.

Throughout this section we use the following quantities repeatedly: for $x, t \in S$, let
\begin{equation}
    \label{eq:yChangeVar}
        y(x,t)=t+x- \frac{\varepsilon\gamma\beta}{2}|x|^{\beta-2}x,
    \end{equation}
    and 
    \begin{align}
        \label{eq:MALAg}
        \begin{split}
        g(x,t)=&-\gamma\left(1-\frac{\beta}{2}\right)(|y(x,t)|^\beta-|x|^\beta)-\frac{\varepsilon}{8}(\gamma\beta)^2\left(|y(x,t)|^{2\beta-2}-|x|^{2\beta-2}\right)\\
        &-\frac{\gamma\beta}{2}\left(|y(x,t)|^{\beta-2}-|x|^{\beta-2}\right)\langle x, y(x,t)\rangle.
        \end{split}
    \end{align}
With this definition for $y(x,t)$, we have
\begin{equation}
\label{eq:innerPrody}
        \langle x ,y(x,t)\rangle=|x|^2-\frac{\varepsilon\gamma\beta}{2}|x|^{\beta} + \langle t,x \rangle.
\end{equation}

\begin{lemma}
\label{lem:equivalentForIntA}
    The necessary condition \eqref{eq:intAto1} in Lemma~\ref{lem:rto0} is satisfied if and only if, for almost all $t\in\mathbb{R}^d$,
    \begin{equation}
        \label{MALAprooflimitge0}
        \liminf_{|x|\to\infty}g(x,t)\ge0.
    \end{equation}
\end{lemma}

\begin{proof}
Given the choice of target \eqref{eq:MALAtarget} and the corresponding MALA proposal \eqref{eq:MALAproposal}, the density \eqref{eq:acceptanceDensity} of the acceptance part of the Markov transition kernel $K(x,dy)$ becomes

    \begin{align*}
        a(x,y)=\min&\Bigg\{1,e^{-\gamma(|y|^\beta-|x|^\beta)-\frac{1}{2\varepsilon}\left(\left\lvert x-y+\frac{\varepsilon\gamma\beta}{2}|y|^{\beta-2}y\right\rvert^2-\left\lvert y-x+\frac{\varepsilon\gamma\beta}{2}|x|^{\beta-2}x\right\rvert^2\right)}\Bigg\}Ce^{-\frac{1}{2\varepsilon}\left\lvert y-x+\frac{\varepsilon\gamma\beta}{2}|x|^{\beta-2}x\right\rvert^2}.
    \end{align*}

     The term coming from the ratio $J(x|y)/J(y|x)$ can be rewritten as
    \begin{align*}
        &\left\lvert x-y+\frac{\varepsilon\gamma\beta}{2}|y|^{\beta-2}y\right\rvert^2-\left\lvert y-x+\frac{\varepsilon\gamma\beta}{2}|x|^{\beta-2}x\right\rvert^2\\
        &=|x-y|^2+\left(\frac{\varepsilon\gamma\beta}{2}\right)^2|y|^{2\beta-2}+\varepsilon\gamma\beta|y|^{\beta-2}(\langle x, y\rangle-|y|^2)\\
        &\quad-|y-x|^2-\left(\frac{\varepsilon\gamma\beta}{2}\right)^2|x|^{2\beta-2}-\varepsilon\gamma\beta|x|^{\beta-2}(\langle x, y\rangle-|x|^2)\\
        &=\left(\frac{\varepsilon\gamma\beta}{2}\right)^2\left(|y|^{2\beta-2}-|x|^{2\beta-2}\right)-\varepsilon\gamma\beta\left(|y|^\beta-|x|^{\beta}\right)+\varepsilon\gamma\beta\left(|y|^{\beta-2}-|x|^{\beta-2}\right)\langle x, y\rangle.
    \end{align*}
    Thus,
    \begin{align*}
        a(x,y)=\min\Bigg\{1,\exp\Bigg\{&-\gamma\left(1-\frac{\beta}{2}\right)(|y|^\beta-|x|^\beta)-\frac{\varepsilon}{8}(\gamma\beta)^2\left(|y|^{2\beta-2}-|x|^{2\beta-2}\right)\\
        &-\frac{\gamma\beta}{2}\left(|y|^{\beta-2}-|x|^{\beta-2}\right)\langle x, y\rangle\Bigg\}\Bigg\}
        \cdot C\exp\left\{-\frac{1}{2\varepsilon}\left\lvert y-x+\frac{\varepsilon\gamma\beta}{2}|x|^{\beta-2}x\right\rvert^2\right\}.
    \end{align*}

    Applying the change of variables
        $t=y-x+\frac{\varepsilon\gamma\beta}{2}|x|^{\beta-2}x$,
we obtain
 \begin{align}
 \label{eq:intaWithJhat}
 \begin{split}        \int_Sa(x,y)dy&=\int_S\min\Bigg\{1,\exp\Bigg\{-\gamma\left(1-\frac{\beta}{2}\right)(|y(x,t)|^\beta-|x|^\beta)\\
        &\quad-\frac{\varepsilon}{8}(\gamma\beta)^2\left(|y(x,t)|^{2\beta-2}-|x|^{2\beta-2}\right)-\frac{\gamma\beta}{2}\left(|y(x,t)|^{\beta-2}-|x|^{\beta-2}\right)\langle x, y(x,t)\rangle\Bigg\}\Bigg\}\, \hat J(t)dt,\\
        &=\int_S\min\left\{1,\exp\left\{g(x,t)\right\}\right\} \hat J(t)dt,
 \end{split}
    \end{align}
    where $y(x,t)$ is given by \eqref{eq:yChangeVar}, and $\hat J(t)=C\exp\left\{-\frac{1}{2\varepsilon}|t|^2\right\}$ is a probability density. By dominated convergence, the necessary condition \eqref{eq:intAto1} is satisfied if and only if 
    \begin{equation*}
        \lim_{|x|\to\infty}\min\left\{1,\exp\left\{g(x,t)\right\}\right\}=1,
    \end{equation*}
    which is equivalent to \eqref{MALAprooflimitge0}.
\end{proof}

To study the limit of $g(x,t)$, as $|x|\to\infty$, for different values of $\beta$, we analyse the behaviour of $|y(x,t)|^\alpha-|x|^\alpha$ as $|x|\to\infty$ with $\alpha=\beta,\,2\beta-2$, and $\beta-2$. For this purpose, observe that, from the definition \eqref{eq:yChangeVar} of $y(x,t)$,
\begin{align}
\label{eq:y2}
\begin{split}
    |y(x,t)|^2&=|t|^2+\left(1-\frac{\varepsilon\gamma\beta}{2}|x|^{\beta-2}\right)^2|x|^2+ 2\left(1-\frac{\varepsilon\gamma\beta}{2}|x|^{\beta-2}\right)\langle t, x\rangle\\
       &=|x|^2\Bigg(1-\varepsilon\gamma\beta|x|^{\beta-2}+\left(\frac{\varepsilon\gamma\beta}{2}\right)^2|x|^{2(\beta-2)}+\frac{|t|^2}{|x|^2}+2\left\langle\frac{t}{|x|},\frac{x}{|x|}\right\rangle-\varepsilon\gamma\beta|x|^{\beta-3}\left\langle t, \frac{x}{|x|}\right\rangle\Bigg).
\end{split}
\end{align}
It will also be useful to consider the Taylor expansion, for $|s|<1$,
 \begin{equation}
   \label{eq:Taylor}
       (1+s)^{\alpha}=1+\alpha s+o(s).
   \end{equation}

\begin{lemma}
    \label{lem:beta0to1}
        Let $0<\beta<1$. Then, \eqref{eq:intAto1} holds. However, there does not exist a function $U:\mathbb{R}^d\to[0,+\infty)$ such that \eqref{eq:intexpUato0} holds.
\end{lemma}

\begin{proof}
We will use Lemma~\ref{lem:equivalentForIntA} to show that \eqref{eq:intAto1} holds. For $0<\beta<1$, as $|x|\to\infty$, \eqref{eq:y2} behaves as
    \begin{equation*}
        |y(x,t)|^2=|x|^2\left(1+2\left\langle \frac{t}{|x|},\frac{x}{|x|}\right\rangle+o(|x|^{-1})\right).
    \end{equation*}
    Using the Taylor expansion \eqref{eq:Taylor} with $s=2\langle t ,x\rangle/|x|^2+o(|x|^{-1})$, then for $\alpha=\beta/2$ we obtain
    \begin{equation*}
        |y(x,t)|^\beta=|x|^\beta\left(1+\beta \left\langle\frac{t}{|x|},\frac{x}{|x|}\right\rangle+o(|x|^{-1})\right),
    \end{equation*}
    and for $\alpha=(\beta-2)/2$,
    \begin{equation*}
        |y(x,t)|^{\beta-2}=|x|^{\beta-2}\left(1+(\beta-2) \left\langle\frac{t}{|x|},\frac{x}{|x|}\right\rangle+o(|x|^{-1})\right).
    \end{equation*}
    Note that the term $-\frac{\varepsilon}{8}(\gamma\beta)^2\left(|y(x,t)|^{2\beta-2}-|x|^{2\beta-2}\right)$ in \eqref{eq:MALAg} is negligible as $|x|\to\infty$ for the values of $\beta$ considered here. Recalling the inner product \eqref{eq:innerPrody}, the limit \eqref{MALAprooflimitge0} in Lemma~\ref{lem:equivalentForIntA} becomes
 \begin{align*}
        &\liminf_{|x|\to\infty}\Bigg[-\gamma\left(1-\frac{\beta}{2}\right)(|y(x,t)|^\beta-|x|^\beta)-\frac{\gamma\beta}{2}\left(|y(x,t)|^{\beta-2}-|x|^{\beta-2}\right)\langle x, y(x,t)\rangle\Bigg]\\
        & \quad =\lim_{|x|\to\infty}\Bigg[-\gamma\left(1-\frac{\beta}{2}\right)\left( \beta |x|^{\beta-1}\left\langle t,\frac{x}{|x|}\right\rangle\right) -\frac{\gamma\beta}{2}\left( (\beta-2) |x|^{\beta-3}\left\langle t, \frac{x}{|x|}\right\rangle\right)\left(|x|^2-\frac{\varepsilon\gamma\beta}{2}|x|^{\beta} + \langle t,x \rangle\right)\Bigg]\\
        & \quad =\lim_{|x|\to\infty}\Bigg[\varepsilon\left(\frac{\gamma\beta}{2}\right)^2(\beta-2)|x|^{2\beta-3}\left\langle t, \frac{x}{|x|}\right\rangle-\frac{\gamma\beta}{2}(\beta-2)|x|^{\beta-2}\left(\left\langle t, \frac{x}{|x|}\right\rangle\right)^2 \Bigg]=0.
        \end{align*}
         Thus, $\lim _{|x| \to \infty} g(x,t) \geq 0$ for almost all $t\in\mathbb{R}^d$, and by Lemma~\ref{lem:equivalentForIntA} we have $\lim_{|x|\to\infty}\int_Sa(x,y)dy=1$.

We now proceed by proving by contradiction that a function $U$ that satisfies \eqref{eq:intexpUato0} cannot exist. Assume that $U:S\to [0,\infty)$ satisfies \eqref{eq:intexpUato0}. By applying the change of variable $w=y-x$ this is equivalent to
\begin{align}
    \label{eq:MALAlimBetalessthan1}
    \begin{split}
        0&=\lim_{|x|\to\infty}\int_Se^{U(x+w)-U(x)}a(x,x+w)dw=\lim_{|x|\to\infty}e^{-U(x)}\int_Se^{U(x+w)}\varpi(x,x+w)J(x+w|x)dw,
    \end{split}        
    \end{align}
    where
      \begin{align*}
       J(x+w|x)&=Ce^{-\frac{1}{2\varepsilon}\left\lvert w+\frac{\varepsilon\gamma\beta}{2}|x|^{\beta-2}x\right\rvert^2}=Ce^{-\frac{|w|^2}{2\varepsilon}-\frac{\gamma\beta}{2}|x|^{\beta-1}\left\langle w, \frac{x}{|x|}\right\rangle-\frac{\varepsilon\gamma^2\beta^2}{8}|x|^{2\beta-2}}.
   \end{align*}

   Note that for a fixed $w\in S$, as $|x|\to \infty$, 
   \begin{equation}
   \label{eq:MALAlimJ}
       \lim_{|x|\to\infty}J(x+w|x)=Ce^{-\frac{|w|^2}{2\varepsilon}}.
   \end{equation}

    Fix a direction $v\in S$, with $|v|=1$, and let $x=\rho v$. Then, if \eqref{eq:MALAlimBetalessthan1} holds, we have
    \begin{equation}
    \label{eq:MALAradialLim}
        \lim_{\rho\to+\infty}e^{-U(\rho v)}\int_Se^{U(\rho v+w)}\varpi(\rho v,\rho v+w)J(\rho v+w|\rho v)dw=0.
    \end{equation}

The following type of construction, and calculations following it, will be used multiple times in the paper. Define the terms in the sequence $\{U_n\}_{n\in\mathbb{N}}\subset [0,\infty)$ according to $U_n = \inf_{\rho \ge n} U(\rho v)$. For a fixed $\varepsilon>0$, for every $n\in\mathbb{N}$ there exists $\rho_n\ge n$ such that 
\begin{equation}
\label{eq:MALArhoclosetoinf}
    0 \le U(\rho_n v) < U_n+\varepsilon.
\end{equation}
We define a sequence $\{\rho_n\}_{n\in\mathbb{N}}\subset[1,\infty)$ with terms $\rho_n$ that satisfy \eqref{eq:MALArhoclosetoinf} for each $n$. By construction, $\lim_{n\to\infty}\rho_n=+\infty$, therefore  \eqref{eq:MALAradialLim} implies
\begin{equation}
\label{eq:MALAlimrhon}
    \lim_{n\to\infty}e^{-U(\rho_n v)}\int_Se^{U(\rho_n v+w)}\varpi(\rho_n v,\rho_n v+w)J(\rho_n v+w|\rho_n v)dw=0.
\end{equation}
Define the set
\begin{equation}
\label{MALA:defH}
    H=\{w\,:\,\langle w,v\rangle \ge 0\}\subset S.
\end{equation}
Observe that if $w\in H$, then for each $n$,
\begin{equation*}
    |\rho_n v+w|=\sqrt{\rho_n^2+|w|^2+2\rho_n\langle w,v\rangle }\ge |\rho_n|\ge n 
\end{equation*}
and therefore, from the definition of $U_n$,
\begin{equation}
\label{eq:MALAineqUn}
    U(\rho_n v+w)\ge U_n.
\end{equation}

Recall that $\varpi(x,y)\le1$ for all $x,y\in S$. Then
\begin{align}
\label{MALA:beta1:varpiLeJ}
    \int_H\varpi(\rho_nv,\rho_nv+w)J(\rho_nv+w|\rho_nv)dw\le \int_HJ(\rho_nv+w|\rho_nv)dw.
\end{align}

Observe that the sequence of measurable functions $\{f_n\}_{n\in\mathbb{N}}$ defined as $f_n(w)=J(\rho_nv+w|\rho_nv)\cdot I\{w\in H\}$ satisfies
\begin{align*}
    |f_n(w)|&=Ce^{-\frac{1}{2\varepsilon}\left\lvert w+\frac{\varepsilon\gamma\beta}{2}\rho_n^{\beta-1}v\right\rvert^2}\cdot I\{w\in H\}=Ce^{-\frac{1}{2\varepsilon}\left( |w|^2+\varepsilon\gamma\beta\rho_n^{\beta-1}\langle w,v\rangle + \left(\frac{\varepsilon\gamma\beta}{2}\right)^2\rho_n^{2\beta-2}\right)}\cdot I\{w\in H\}\\
    &\le Ce^{-\frac{|w|^2}{2\varepsilon}}\cdot I\{w\in H\},
\end{align*}
because, by definition, $|v|=1$, $\rho_n\ge1$, and $\langle w,v\rangle\ge0$ for all $w\in H$. Note that the function $w\mapsto Ce^{-\frac{|w|^2}{2\varepsilon}}\cdot I\{w\in H\}$ is integrable and, in particular, 
\begin{align*}
    \int_Sf_n(w)dw\le\int_HCe^{-\frac{|w|^2}{2\varepsilon}}=\frac{1}{2},
\end{align*}
since $w\mapsto Ce^{-\frac{|w|^2}{2\varepsilon}}$ is the density of a centered $d$-dimensional Gaussian distribution on $S$, and $H$ is a half-space in $\mathbb{R}^d$ defined via a hyperplane ($\{w\,:\,\langle w,v\rangle =0\}$) passing through the origin.
Besides, by \eqref{eq:MALAlimJ}, the sequence of functions $\{f_n\}$ converges pointwise to $Ce^{-\frac{|w|^2}{2\varepsilon}}$. By the dominated convergence theorem, it follows that
\begin{align*}
\lim_{n\to\infty}\int_Sf_n(w)dw&=\int_S\lim_{n\to\infty}f_n(w)dw=\int_S Ce^{-\frac{|w|^2}{2\varepsilon}}\cdot I\{w\in H\}dw =\frac{1}{2},
\end{align*}
that is,
\begin{align}
\label{MALA:beta1:limJ}
    \lim_{n\to\infty}\int_HJ(\rho_nv+w|\rho_nv)dw=\frac{1}{2}.
\end{align}

Equations \eqref{MALA:beta1:varpiLeJ} and \eqref{MALA:beta1:limJ} combined imply
\begin{equation}
    \label{ineq_le12}
    \limsup_{n\to\infty}\int_H\varpi(\rho_nv,\rho_nv+w)J(\rho_nv+w|\rho_nv)dw\le \frac{1}{2}.
\end{equation}
In the other direction, using $\lim _{|x| \to \infty} \int _S a (x,y) dy=1$, we obtain the asymptotic lower bound
\begin{align*}
    &\liminf_{n\to\infty}\int_H\varpi(\rho_nv,\rho_nv+w)J(\rho_nv+w|\rho_nv)dw= \liminf_{n\to\infty}\int_H a(\rho_nv,\rho_nv+w)dw\\
    &\qquad=\liminf_{n\to\infty}\left(\int_Sa(\rho_nv,\rho_nv+w)dw-\int_{S\setminus H}a(\rho_nv,\rho_nv+w)dw\right)\\
    & \qquad \ge \liminf_{n\to\infty}\int_Sa(\rho_nv,\rho_nv+w)dw + \liminf_{n\to\infty}\left(-\int_{S\setminus H}a(\rho_nv,\rho_nv+w)dw\right)\\
    &\qquad=1-\limsup\int_{S\setminus H}\varpi(\rho_nv,\rho_nv+w)J(\rho_nv+w|\rho_nv)dw\\
    &\qquad \ge1-\limsup_{n\to\infty}\int_{S\setminus H}J(\rho_nv+w|\rho_nv)dw=\liminf_{n\to\infty}\int_{H}J(\rho_nv+w|\rho_nv)dw=\frac{1}{2},
\end{align*}
where we used Equation \eqref{MALA:beta1:limJ}. Therefore,
\begin{align*}
    \liminf_{n\to\infty}\int_H\varpi(\rho_nv,\rho_nv+w)J(\rho_nv+w|\rho_nv)dw \ge \frac{1}{2},
\end{align*}
and recalling \eqref{ineq_le12} we obtain that 
\begin{equation}
\label{eq:MALAlimVarpi}
\lim_{n\to\infty}\int_H\varpi(\rho_nv,\rho_nv+w)J(\rho_nv+w|\rho_nv)dw=\frac{1}{2}.
\end{equation}

Thus, from \eqref{eq:MALArhoclosetoinf}, \eqref{eq:MALAineqUn} and \eqref{eq:MALAlimVarpi}, it follows that
\begin{align*}
    &\lim_{n\to\infty}\int_Se^{U(\rho_n v+w)-U(\rho_n v)}\varpi(\rho_n v,\rho_n v+w)J(\rho_n v+w|\rho_n v)dw\\
    & \qquad \ge \lim_{n\to\infty}e^{-U_n-\varepsilon}\int_Se^{U(\rho_n v+w)}\varpi(\rho_n v,\rho_n v+w)J(\rho_n v+w|\rho_n v)dw\\
    & \qquad \ge\lim_{n\to\infty}e^{-U_n-\varepsilon}\int_He^{U(\rho_n v+w)}\varpi(\rho_n v,\rho_n v+w)J(\rho_n v+w|\rho_n v)dw\\
    & \qquad \ge\lim_{n\to\infty}e^{-U_n-\varepsilon}e^{U_n}\int_H\varpi(\rho_n v,\rho_n v+w)J(\rho_n v+w|\rho_n v)dw\\
    & \qquad =e^{-\varepsilon}\lim_{n\to\infty}\int_H\varpi(\rho_n v,\rho_n v+w)J(\rho_n v+w|\rho_n v)dw=\frac{e^{-\varepsilon}}{2}>0.
\end{align*}
This contradicts \eqref{eq:MALAlimrhon}. Using Lemma \ref{lem:rto0} we conclude that a function $U$ satisfying \eqref{eq:intexpUato0} cannot exist.
\end{proof}

\begin{lemma}
\label{lem:beta1}
    Let $\beta=1$.
    Then, \eqref{eq:intAto1} holds. However, there does not exist a function $U:\mathbb{R}^d\to[0,+\infty)$ such that \eqref{eq:intexpUato0} holds.
\end{lemma}
\begin{proof}
    When $\beta=1$, $g(x,t)$, as defined in \eqref{eq:MALAg}, becomes
 \begin{equation}
 \label{eq:MALAlimbeta1}
        g(x,t)=-\frac{\gamma}{2}(|y(x,t)|-|x|)-\frac{\gamma}{2}\left(|y(x,t)|^{-1}-|x|^{-1}\right)\langle x, y(x,t)\rangle.
        \end{equation}
Moreover, from \eqref{eq:y2}, as $|x|\to\infty$ and with $\beta = 1$, $|y(x,t)|^2$ behaves as
    \begin{equation*}
        |y(x,t)|^2=|x|^2\left(1+|x|^{-1}\left(-\varepsilon\gamma+2\left\langle t,\frac{x}{|x|}\right\rangle\right)+o(|x|^{-1})\right).
    \end{equation*}
    Using the Taylor expansion \eqref{eq:Taylor} with $s=|x|^{-1}\left(-\varepsilon\gamma+2\langle t, x\rangle/|x|\right)+o(|x|^{-1})$, for $\alpha=1/2$ we have
    \begin{equation*}
        |y(x,t)|=|x|\left(1+\frac{1}{2} |x|^{-1} \left(-\varepsilon\gamma+2\left\langle t,\frac{x}{|x|}\right\rangle\right)+o(|x|^{-1})\right),
    \end{equation*}s
    and for $\alpha=-1/2$,
      \begin{equation*}
        |y(x,t)|^{-1}=|x|^{-1}\left(1-\frac{1}{2} |x|^{-1} \left(-\varepsilon\gamma+2\left\langle t,\frac{x}{|x|}\right\rangle\right)+o(|x|^{-1})\right).
    \end{equation*}
    Moreover, recalling the inner product \eqref{eq:innerPrody}, the limit in \eqref{MALAprooflimitge0} becomes
    \begin{align*}
    \label{eq:MALAlimBeta1}
         &\liminf_{|x|\to\infty}\Bigg[-\frac{\gamma}{2}\left(-\frac{\varepsilon\gamma}{2}+\left\langle t, \frac{x}{|x|}\right\rangle\right)-\frac{\gamma}{2}|x|^{-2}\left(\frac{\varepsilon\gamma}{2}-\left\langle t, \frac{x}{|x|}\right\rangle\right)\left(|x|^2+\langle t, x \rangle -\frac{\varepsilon\gamma}{2}|x|\right)\Bigg]\\
         & \quad =\lim_{|x|\to\infty}\frac{\gamma}{2}|x|^{-1}\left(\frac{\varepsilon\gamma}{2}-\left\langle t, \frac{x}{|x|}\right\rangle\right)^2 =0.
    \end{align*}
    By Lemma~\ref{lem:equivalentForIntA}, this implies that $\lim_{|x|\to\infty}\int_S a(x,y)dy=1$.

To show that there cannot exist a function $U$ satisfying \eqref{eq:intexpUato0}, we now follow the same strategy as in Lemma \ref{lem:beta0to1}. Suppose that there is such a function. With the change of variable $w=y-x$, the proposal density takes the form
\begin{equation*}
       J(x+w|x)=C\exp\left\{-\frac{1}{2\varepsilon}\left\lvert w+\frac{\varepsilon\gamma}{2}\frac{x}{|x|}\right\rvert^2\right\}.
   \end{equation*}
Because the proposal depends on $x/|x|$, fix a direction $v \in S$, with $|v|=1$ and consider $x = \rho v$, $\rho \in \mathbb{R}$. Then,
\begin{equation*}
    \lim_{\rho\to\infty}J(\rho v+w|\rho v)=C\exp\left\{-\frac{1}{2\varepsilon}\left\lvert w+\frac{\varepsilon\gamma}{2}v\right\rvert^2\right\}.
\end{equation*}
Similar to Lemma \ref{lem:beta0to1}, define the sequence $\{ U_n \} _{n \in \mathbb{N}}$ according to $U_n = \inf_{\rho \geq n} U(\rho v)$. For a fixed $\varepsilon >0$, we can extract a sequence $\{ \rho _n\} _{n \in \mathbb{N}}$ such that $\rho _n \geq n$ and $0 \leq U(\rho_n v) < U_n + \varepsilon$. Then $\rho _n \to +\infty$ as $n \to \infty$, and by the assumption on $U$,
\begin{equation}
\label{eq:limU_beta1}
    \lim _{n \to \infty} e^{-U(\rho _n v)} \int _S e^{U(\rho_n v + w)} \varpi (\rho_n v, \varpi _n v + w) J (\rho_n v +w | \rho _n v) dw =0,
\end{equation}
with $\varpi$ and $J$ associated with the MALA density with $\beta =1$.

We will now use an argument analogous to that used in Lemma~and \ref{lem:beta0to1}. Define $H$ as in \eqref{MALA:defH} and, for a fixed direction $v$, define the sequence of measurable functions $\{f_n\}$ as
\begin{align*}
    f_n(w)=J(\rho_nv+w|\rho_nv)\cdot I\{w\in H\}.
\end{align*}
Note that
\begin{align*}
    f_n(w)=Ce^{-\frac{1}{2\varepsilon}\left\lvert w+\frac{\varepsilon\gamma}{2}v\right\rvert^2}\cdot I\{w\in H\},
\end{align*}
where $w\mapsto Ce^{-\frac{1}{2\varepsilon}\left\lvert w+\frac{\varepsilon\gamma}{2}v\right\rvert^2}$ is the density of a $d$-dimensional Gaussian distribution centered in $-\frac{\varepsilon\gamma}{2}v$, which is $\neq0$. Because $H$ is the hyper-space defined via a hyperplane passing through the origin,
\begin{align*}
    \int_Sf_n(w)=\int_HCe^{-\frac{1}{2\varepsilon}\left\lvert w+\frac{\varepsilon\gamma}{2}v\right\rvert^2}dw=c\in (0,1).
\end{align*}
Recalling that $0\le\varpi(x,y)\le 1$ for all $x,y\in S$, it follows that
\begin{equation}
\label{ineq_le_c}
\begin{split}
    \limsup_{n\to\infty}\int_H\varpi(\rho_nv,\rho_nv+w)J(\rho_nv+w|\rho_nv)dw&\le \lim_{n\to\infty}\int_HJ(\rho_nv+w|\rho_nv)dw\\
    &=\lim_{n\to\infty}\int_Sf_n(w)dw=\lim_{n\to\infty}c=c.
\end{split}
\end{equation}
Besides, since $1=\lim_{|x|\to\infty}\int_S a(x,y)dy=\lim_{|x|\to\infty}\int_S \varpi(x,y)J(y|x)dy$,
\begin{align*}
    &\liminf_{n\to\infty}\int_H\varpi(\rho_nv,\rho_nv+w)J(\rho_nv+w|\rho_nv)dw=\liminf_{n\to\infty}\int_H a(\rho_nv,\rho_nv+w)dw\\
    &\qquad=\liminf_{n\to\infty}\left(\int_S a(\rho_nv,\rho_nv+w)dw-\int_{S\setminus H} a(\rho_nv,\rho_nv+w)dw\right)\\
    &\qquad \ge \liminf_{n\to\infty}\int_S a(\rho_nv,\rho_nv+w)dw+\liminf_{n\to\infty}\left(-\int_{S\setminus H} a(\rho_nv,\rho_nv+w)dw\right)\\
    &\qquad=1-\limsup_{n\to\infty}\int_{S\setminus H}\varpi(\rho_nv,\rho_nv+w)J(\rho_nv+w|\rho_nv)dw\\
    &\qquad\ge 1-\limsup_{n\to\infty}\int_{S\setminus H}J(\rho_nv+w|\rho_nv)dw=\lim_{n\to\infty}\int_HJ(\rho_nv+w|\rho_nv)dw=c,
\end{align*}
which, together with \eqref{ineq_le_c}, leads to
\begin{align*}
    \lim_{n\to\infty}\int_H\varpi(\rho_nv,\rho_nv+w)J(\rho_nv+w|\rho_nv)dw  =c.
\end{align*}

It follows that
\begin{align*}
    & \lim_{n\to\infty}\int_Se^{U(\rho_n v+w)-U(\rho_n v)}\varpi(\rho_n v,\rho_n v+w)J(\rho_n v+w|\rho_n v)dw\\
    & \qquad \ge \lim_{n\to\infty}e^{-U_n-\varepsilon}\int_Se^{U(\rho_n v+w)}\varpi(\rho_n v,\rho_n v+w)J(\rho_n v+w|\rho_n v)dw\\
    & \qquad \ge\lim_{n\to\infty}e^{-U_n-\varepsilon}\int_He^{U(\rho_n v+w)}\varpi(\rho_n v,\rho_n v+w)J(\rho_n v+w|\rho_n v)dw\\
    & \qquad \ge\lim_{n\to\infty}e^{-U_n-\varepsilon}e^{U_n}\int_H\varpi(\rho_n v,\rho_n v+w)J(\rho_n v+w|\rho_n v)dw\\
    & \qquad =e^{-\varepsilon}\lim_{n\to\infty}\int_H\varpi(\rho_n v,\rho_n v+w)J(\rho_n v+w|\rho_n v)dw=ce^{-\varepsilon}.
\end{align*}
Since $c e^{-\varepsilon} >0$, this contradicts the assumption that $U$ satisfies \eqref{eq:limU_beta1}. By extension, $U$ cannot satisfy \eqref{eq:intexpUato0}, which completes the proof.

\end{proof}

\begin{lemma}
\label{lem:beta1to2}
    Let $1<\beta<2$. Then, $\lim_{|x|\to\infty}\int_Sa(x,y)dy=1$.
\end{lemma}

\begin{proof}
    To prove the claim we will once again rely on Lemma~\ref{lem:equivalentForIntA}. From the definition of $y(x,t)$ and \eqref{eq:y2}, when $1<\beta<2$ and as $|x|\to\infty$,
    \begin{equation*}
        |y(x,t)|^2=|x|^2\left(1-\varepsilon\gamma\beta|x|^{\beta-2}+\left(\frac{\varepsilon\gamma\beta}{2}\right)^2|x|^{2\beta-4}+2\left\langle t,\frac{x}{|x|}\right\rangle |x|^{-1}-\varepsilon\gamma\beta|x|^{\beta-3}+o\left(|x|^{\beta-3}\right)\right).
    \end{equation*}
   Using the Taylor expansion \eqref{eq:Taylor} with $s=-\varepsilon\gamma\beta|x|^{\beta-2}+\left(\frac{\varepsilon\gamma\beta}{2}\right)^2|x|^{2\beta-4}+2\left\langle t,\frac{x}{|x|}\right\rangle |x|^{-1}$,
    we obtain that for $\alpha=\beta/2$,
    \begin{equation*}
        |y(x,t)|^\beta=|x|^\beta\left(1-\frac{\varepsilon\gamma\beta^2}{2}|x|^{\beta-2}+\frac{\varepsilon^2\gamma^2\beta^3}{8}|x|^{2\beta-4}+\beta\left\langle t,\frac{x}{|x|}\right\rangle |x|^{-1}-\frac{\varepsilon\gamma\beta^2}{2}|x|^{\beta-3}+o(|x|^{\beta-3})\right),
    \end{equation*}
    for $\alpha=\beta-1$,
    \begin{equation*}
        \begin{split}
        |y(x,t)|^{2\beta-2}=|x|^{2\beta-2}\Bigg(1-\varepsilon\gamma\beta(\beta-1)&|x|^{\beta-2}+\frac{\varepsilon^2\gamma^2\beta^2}{4}(\beta-1)|x|^{2\beta-4}\\
        &+2(\beta-1)\left\langle t,\frac{x}{|x|}\right\rangle |x|^{-1}-\varepsilon\gamma\beta(\beta-1)|x|^{\beta-3}+o(|x|^{\beta-3})\Bigg),
        \end{split}
    \end{equation*}
    and for $\alpha=(\beta-2)/2$,
    \begin{equation*}
        \begin{split}
        |y(x,t)|^{\beta-2}=|x|^{\beta-2}\Bigg(1-\frac{\varepsilon\gamma\beta}{2}(\beta-2)|x|^{\beta-2}&+\frac{\varepsilon^2\gamma^2\beta^2}{8}(\beta-2)|x|^{2\beta-4}\\
        &+(\beta-2)\left\langle t,\frac{x}{|x|}\right\rangle |x|^{-1}-\frac{\varepsilon\gamma\beta}{2}(\beta-2)|x|^{\beta-3}+o(|x|^{\beta-3})\Bigg).
        \end{split}
    \end{equation*}
    As a consequence, we obtain the following:
    \begin{equation*}
        \begin{split}
            |y(x,y)|^\beta-|x|^\beta&=-\frac{\varepsilon\gamma\beta^2}{2}|x|^{2\beta-2}+\frac{\varepsilon^2\gamma^2\beta^3}{8}|x|^{3\beta-4}\\
            &\qquad+\beta\left\langle t,\frac{x}{|x|}\right\rangle |x|^{\beta-1}-\frac{\varepsilon\gamma\beta^2}{2}|x|^{2\beta-3}+o(|x|^{2\beta-3})\\
            |y(x,y)|^{2\beta-2}-|x|^{2\beta-2}&=-\varepsilon\gamma\beta(\beta-1)|x|^{3\beta-4}+\frac{\varepsilon^2\gamma^2\beta^2}{4}(\beta-1)|x|^{4\beta-6}\\
            &\qquad +2(\beta-1)\left\langle t,\frac{x}{|x|}\right\rangle |x|^{2\beta-3}-\varepsilon\gamma\beta(\beta-1)|x|^{3\beta-5}+o(|x|^{3\beta-5})\\
            |y(x,y)|^{\beta-2}-|x|^{\beta-2}&=-\frac{\varepsilon\gamma\beta}{2}(\beta-2)|x|^{2\beta-4}+\frac{\varepsilon^2\gamma^2\beta^2}{8}(\beta-2)|x|^{3\beta-6}\\
        &\qquad +(\beta-2)\left\langle t,\frac{x}{|x|}\right\rangle |x|^{\beta-3}-\frac{\varepsilon\gamma\beta}{2}(\beta-2)|x|^{2\beta-5}+o(|x|^{2\beta-5})
        \end{split}
    \end{equation*}
    
    Using \eqref{eq:innerPrody} for the inner product $\langle x, y(x,t) \rangle$, $g(x,t)$ becomes
    \begin{align*}
        g(x,t)&=-\gamma\left(1-\frac{\beta}{2}\right)(|y(x,t)|^\beta-|x|^\beta)-\frac{\varepsilon}{8}(\gamma\beta)^2\left(|y(x,t)|^{2\beta-2}-|x|^{2\beta-2}\right)\\
        &\qquad-\frac{\gamma\beta}{2}\left(|y(x,t)|^{\beta-2}-|x|^{\beta-2}\right)\langle x, y(x,t)\rangle\\
        &=\frac{\gamma}{2}(\beta-2)\left(-\frac{\varepsilon\gamma\beta^2}{2}|x|^{2\beta-2}+\frac{\varepsilon^2\gamma^2\beta^3}{8}|x|^{3\beta-4}+\beta\left\langle t,\frac{x}{|x|}\right\rangle |x|^{\beta-1}-\frac{\varepsilon\gamma\beta^2}{2}|x|^{2\beta-3}+o(|x|^{2\beta-3})\right)\\
        &\qquad -\frac{\varepsilon}{8}(\gamma\beta)^2\Bigg(-\varepsilon\gamma\beta(\beta-1)|x|^{3\beta-4}+\frac{\varepsilon^2\gamma^2\beta^2}{4}(\beta-1)|x|^{4\beta-6}\\
        &\qquad\qquad\qquad\qquad  +2(\beta-1)\left\langle t,\frac{x}{|x|}\right\rangle |x|^{2\beta-3}-\varepsilon\gamma\beta(\beta-1)|x|^{3\beta-5}+o(|x|^{3\beta-5})\Bigg)\\
        &\qquad-\frac{\gamma\beta}{2}\Bigg(-\frac{\varepsilon\gamma\beta}{2}(\beta-2)|x|^{2\beta-4}+\frac{\varepsilon^2\gamma^2\beta^2}{8}(\beta-2)|x|^{3\beta-6}\\
        &\qquad\qquad\qquad\qquad +(\beta-2)\left\langle t,\frac{x}{|x|}\right\rangle |x|^{\beta-3}-\frac{\varepsilon\gamma\beta}{2}(\beta-2)|x|^{2\beta-5}+o(|x|^{2\beta-5})\Bigg)\\
        &\qquad\qquad\qquad\qquad\times\left(|x|^2-\frac{\varepsilon\gamma\beta}{2}|x|^\beta+\left\langle t,\frac{x}{|x|}\right\rangle|x|\right)\\
        &=\Bigg(\cancel{-\frac{\varepsilon\gamma^2\beta^2}{4}(\beta-2)|x|^{2\beta-2}}+\cancel{\frac{\varepsilon^2\gamma^3\beta^3}{16}(\beta-2)|x|^{3\beta-4}}+\cancel{\frac{\gamma\beta}{2}(\beta-2)\left\langle t,\frac{x}{|x|}\right\rangle|x|^{\beta-1}}\\
        &\qquad\qquad\qquad\qquad-\cancel{\frac{\varepsilon\gamma^2\beta^2}{4}(\beta-2)|x|^{2\beta-3}}+o(|x|^{2\beta-3})\Bigg)\\
        &\qquad+\Bigg(\frac{\varepsilon^2\gamma^3\beta^3}{8}(\beta-1)|x|^{3\beta-4}-\frac{\varepsilon^3\gamma^4\beta^4}{32}(\beta-1)|x|^{4\beta-6}\\
        &\qquad\qquad\qquad\qquad-\frac{\varepsilon\gamma^2\beta^2}{4}(\beta-1)\left\langle t,\frac{x}{|x|}\right\rangle|x|^{2\beta-3}+\frac{\varepsilon^2\gamma^3\beta^3}{8}(\beta-1)|x|^{3\beta-5}+o(|x|^{3\beta-5})\Bigg)\\
        &+\Bigg(\cancel{\frac{\varepsilon\gamma^2\beta^2}{4}(\beta-2)|x|^{2\beta-2}}-\cancel{\frac{\varepsilon^2\gamma^3\beta^3}{16}(\beta-2)|x|^{3\beta-4}}\\
        &\qquad\qquad\qquad\qquad\cancel{-\frac{\gamma\beta}{2}(\beta-2)\left\langle t,\frac{x}{|x|}\right\rangle|x|^{\beta-1}}+\cancel{\frac{\varepsilon\gamma^2\beta^2}{4}(\beta-2)|x|^{2\beta-3}}+o(|x|^{2\beta-3})\Bigg)\\
        &\qquad+\left(-\frac{\varepsilon^2\gamma^3\beta^3}{8}(\beta-2)|x|^{3\beta-4}+o(|x|^{3\beta-4})\right)+\left(\frac{\varepsilon\gamma^2\beta^2}{4}(\beta-2)\left\langle t,\frac{x}{|x|}\right\rangle|x|^{2\beta-3}+o(|x|^{2\beta-3})\right)\\
        &=\frac{\varepsilon^2\gamma^3\beta^3}{8}|x|^{3\beta-4}+o(|x|^{3\beta-4}),
        \end{align*}
        since for $\beta\in(1,2)$, $|x|^{2\beta-3}=o(|x|^{3\beta-4})$, $|x|^{4\beta-6}=o(|x|^{3\beta-4})$ and $|x|^{3\beta-5}=o(|x|^{3\beta-4})$.

        Note that $\frac{\varepsilon^2\gamma^3\beta^3}{8}|x|^{3\beta-4}\ge0$, therefore,
        \begin{equation*}
            \liminf_{|x|\to\infty}g(x,t)=\liminf_{|x|\to\infty} \frac{\varepsilon^2\gamma^3\beta^3}{8}|x|^{3\beta-4}\ge0,
        \end{equation*}
        and by Lemma~\ref{lem:equivalentForIntA} we obtain the desired limit.
\end{proof}

\begin{lemma}
\label{lem:beta2}
    Let $\beta=2$. Then, $\lim_{|x|\to\infty}\int_Sa(x,y)dy=1$ if and only if $\varepsilon\gamma<2$.
\end{lemma}

\begin{proof}  
    When $\beta=2$, $|y(x,t)|^2$ becomes (see \eqref{eq:y2})
    \begin{equation*}
        |y(x,t)|^2=|x|^2\left(1-2\varepsilon\gamma+\left(\varepsilon\gamma\right)^2+\frac{|t|^2}{|x|^2}+2(1-\varepsilon\gamma)\left\langle \frac{t}{|x|},\frac{x}{|x|}\right\rangle\right).
    \end{equation*}
    If we restrict to the two cases $\varepsilon \gamma < 2$ and $\varepsilon \gamma >2$, the limit of $g(x,t)$, as defined in \eqref{eq:MALAg}, simplifies to        
    \begin{align*}
        \lim_{|x|\to\infty}-\frac{\varepsilon}{2}\gamma^2(|y(x,t)|^2-|x|^2)&=-\frac{\varepsilon^2\gamma^3}{2}(\varepsilon\gamma-2)\lim_{|x|\to\infty}|x|^2=\begin{cases}
            +\infty,\quad\text{if }\varepsilon\gamma <2,\\
            -\infty,\quad\text{if }\varepsilon\gamma >2.\\
        \end{cases}
    \end{align*}
    Thus, when $\beta=2$ and $\varepsilon\gamma<2$, \eqref{MALAprooflimitge0} holds and, by Lemma~\ref{lem:equivalentForIntA}, $\lim_{|x|\to\infty}\int_Sa(x,y)dy=1$. 
On the other hand, when $\beta=2$ and $\varepsilon\gamma>2$, the limit in \eqref{MALAprooflimitge0} is negative, and the claim again follows from Lemma~\ref{lem:equivalentForIntA}.

    It remains to consider the case $\beta =2$ and $\varepsilon\gamma=2$. With these parameter values,
    \begin{equation*}
       |y(x,t)|^2-|x|^2=|t|^2-2\langle t, x\rangle, 
    \end{equation*}
    and the sign of $\lim_{|x|\to\infty}-\frac{\varepsilon}{2}\gamma^2(|y(x,t)|^2-|x|^2)$ depends on the direction $x/|x|$. Take a direction $v\in\mathbb{R}^d$ with $|v|=1$, set $x=\rho v$ with $\rho\ge0$, and consider the limit in the direction $v$. Let $H_v^-$ and $H_v^+$ be the half-spaces of $\mathbb{R}^d$ defined as $H_v^-=\{t:\langle t, v\rangle <0\}$ and $H_v^+=\{t:\langle t, v\rangle >0\}$. Then,
    \begin{equation}
    \label{lim_g}
        \lim_{\rho\to+\infty}-\frac{\varepsilon}{2}\gamma^2(|y(\rho v,t)|^2-|\rho v|^2)=\lim_{\rho\to\infty}-\frac{\varepsilon}{2}\gamma^2(|t|^2-2\rho \langle t, v\rangle)=\begin{cases}
            -\infty\quad\text{if }t\in H_v^-,\\
             +\infty\quad\text{if }t\in H_v^+.
        \end{cases}
    \end{equation}
    This implies that 
    \begin{align*}
        \liminf_{|x|\to\infty}g(x,t)=\liminf_{|x|\to\infty}-\frac{\varepsilon}{2}\gamma^2(|y(x,t)|^2-|x|^2)=-\infty.
    \end{align*}
    Consequently, by Lemma~\ref{lem:equivalentForIntA}, we conclude that $\lim_{|x|\to\infty}\int_Sa(x,y)dy
\neq1$ when $\beta=2$ and $\varepsilon\gamma=2$, which completes the proof.

On an interesting note (not necessary for the proof), we now show that the average acceptance probability converges to $\frac{1}{2}$, i.e.,
\[\lim_{|x|\to\infty}\int_Sa(x,y)dy=\frac{1}{2}.\] Recall the expression \eqref{eq:intaWithJhat} for $\int _S a (x,y)dy$ and the limit \eqref{lim_g}, and observe that $\int_{H_v^-}\hat J(t)dt=\int_{H_v^+}\hat J(t)dt=1/2$. For the choice $x = \rho v$, this yields 
    \begin{align*}
        &\lim_{\rho\to\infty}\int_Sa(\rho v,y)dy=\lim_{\rho\to\infty}\int_S\min\left\{1,\exp\left\{g(\rho v,t)\right\}\right\} \hat J(t)dt\\
        &=\lim_{\rho\to\infty}\Bigg[\int_{H_v^-}\min\left\{1,\exp\left\{g(\rho v,t)\right\}\right\} \hat J(t)dt+\int_{H_v^+}\min\left\{1,\exp\left\{g(\rho v,t)\right\}\right\} \hat J(t)dt\Bigg]=\frac{1}{2}\cdot 0 + \frac{1}{2}\cdot 1 = \frac{1}{2}.
    \end{align*}
    This limit is independent of the direction $v$, therefore, it follows that $\lim_{|x|\to\infty}\int_Sa(x,y)dy=\frac{1}{2}$.
\end{proof}

We now prove the positive part of Proposition \ref{prop:MALA}: there exists a Lyapunov function satisfying the desired properties when either $\beta \in (1,2)$ or $\beta = 2$ and $\varepsilon \gamma <2$. 
\begin{lemma}
\label{lem:Uforbeta1to2}
    If $\beta=2$ and $\varepsilon\gamma<2$, or $1<\beta<2$ there exists a function $U:\mathbb{R}^d\to[0,+\infty)$ that satisfies Assumption~\ref{ass:compactSpaceOrLyapunov}.
\end{lemma}

\begin{proof}
    We will show that the specific choice $U(x)=|x|$ satisfies Assumption~\ref{ass:compactSpaceOrLyapunov} for the given ranges of $\beta, \varepsilon, \gamma$. To start, we note that by continuity $U(x)$ satisfies Property~\ref{ass:supU}.
    
    Next, we prove that Property~\ref{ass:relCompact} holds for this choice of $U$. By Lemma~\ref{lem:rto0}, this property is equivalent to the limits \eqref{eq:intAto1} and \eqref{eq:intexpUato0}, the first of which was shown to hold for the parameter values considered here in Lemmas \ref{lem:beta1to2}-\ref{lem:beta2}. It therefore remains to show that this choice of $U$ also satisfies \eqref{eq:intexpUato0}.

    Because $a(x,y)\le J(y|x)$,
    \begin{align*}
    \begin{split}
        \int_Se^{U(y)-U(x)}a(x,y)dy&\le\int_Se^{U(y)-U(x)}J(y|x)dy\\
        &  =C\int_S\exp\left\{|y|-|x|-\frac{1}{2\varepsilon}\left\lvert y-\left(1-\frac{\varepsilon\gamma\beta}{2}|x|^{\beta-2}\right)x\right\rvert^2\right\}dy\\
        &=C\int_S\exp\left\{\left\lvert t+\left(1-\frac{\varepsilon\gamma\beta}{2}|x|^{\beta-2}\right)x\right\lvert-|x|-\frac{1}{2\varepsilon}\left\lvert t\right\rvert^2\right\}dt\\
        &\le C\int_S\exp\left\{\left\lvert t\right\rvert+\left\lvert 1-\frac{\varepsilon\gamma\beta}{2}|x|^{\beta-2}\right\rvert \left\lvert x\right\lvert-|x|-\frac{1}{2\varepsilon}\left\lvert t\right\rvert^2\right\}dt\\
            & =C\exp\left\{\left(\left\lvert 1-\frac{\varepsilon\gamma\beta}{2}|x|^{\beta-2}\right\rvert-1\right)|x|\right\}
            \cdot\int_S\exp\left\{-\frac{1}{2\varepsilon}|t|^2+|t|\right\}dt,\\
            &=C\exp\left\{\left(\left\lvert 1-\frac{\varepsilon\gamma\beta}{2}|x|^{\beta-2}\right\rvert-1\right)|x|\right\}\cdot C',
        \end{split}
    \end{align*}
    where on the third line we applied the change of variable $t=y-\left(1-\frac{\varepsilon\gamma\beta}{2}|x|^{\beta-2}\right)x$, the inequality on the fourth line follows by the triangular inequality, and $C'\in\mathbb{R}$ is a finite constant that is independent of $x$. 

    Consider the case $1<\beta<2$. When $|x|\to\infty$, $\frac{\varepsilon\gamma\beta}{2}|x|^{\beta-2}\to0$, therefore,
    \begin{equation*}
        \begin{split}
            0&\le\lim_{|x|\to\infty}\int_Se^{U(y)-U(x)}a(x,y)dy\le \lim_{|x|\to\infty}CC'\exp\left\{\left(\left\lvert 1-\frac{\varepsilon\gamma\beta}{2}|x|^{\beta-2}\right\rvert-1\right)|x|\right\}\\
            &= \lim_{|x|\to\infty}CC'\exp\left\{\left(1-\frac{\varepsilon\gamma\beta}{2}|x|^{\beta-2}-1\right)|x|\right\}=\lim_{|x|\to\infty}CC'\exp\left\{-\frac{\varepsilon\gamma\beta}{2}|x|^{\beta-1}\right\}=0,
        \end{split}
    \end{equation*}
    which proves the limit \eqref{eq:intexpUato0} when $1<\beta<2$.

    When $\beta=2$,
    \begin{equation*}
    \begin{split}
        0\le\lim_{|x|\to\infty}\int_Se^{U(y)-U(x)}a(x,y)dy&\le \lim_{|x|\to\infty}CC'\exp\left\{\left(\left\lvert 1-\frac{\varepsilon\gamma\beta}{2}|x|^{\beta-2}\right\rvert-1\right)|x|\right\}\\
        &=\lim_{|x|\to\infty}CC'\exp\left\{\left(\left\lvert 1-\varepsilon\gamma\right\rvert-1\right)|x|\right\}.
    \end{split}
    \end{equation*}
    The limit \eqref{eq:intexpUato0} is satisfied if and only if the limit in the last display is equal to $0$. This holds if and only if $\left\lvert 1-\varepsilon\gamma\right\rvert-1<0$, which is equivalent to $0<\varepsilon\gamma<2$.

    We thus obtain that $U(x)=|x|$ satisfies Property~\ref{ass:relCompact} when $\beta=2$ and $\varepsilon\gamma<2$, or $0<\beta<2$.
    
    We finish the proof by showing that Property~\ref{ass:notNegativeInf} holds for $U = |x|$. Recalling  \eqref{eq:expressionU}, we can rewrite Property~\ref{ass:notNegativeInf} as
    \begin{equation*}
        \inf_{x\in S}-\log\left(\int_Se^{U(y)-U(x)}a(x,y)dy+r(x)\right)>-\infty.
    \end{equation*}
    This is equivalent to
        \begin{equation}
        \label{eq:supLessInfty}
        \sup_{x\in S}\int_Se^{U(y)-U(x)}a(x,y)dy+r(x)<+\infty.
    \end{equation}
    Because \eqref{eq:intexpUato0} holds and $r(x)\in[0,1]$ for all $x\in\mathbb{R}^d$, we have
       \begin{equation*}
       \limsup_{|x|\to\infty}\int_Se^{U(y)-U(x)}a(x,y)dy+r(x)<+\infty.
    \end{equation*}
    Moreover, because $a(x,y),U(x)$ and $r(x)$ are continuous functions, \eqref{eq:supLessInfty} must hold. This proves Property~\ref{ass:notNegativeInf} in Assumption~\ref{ass:compactSpaceOrLyapunov}, which in turn completes the proof.
\end{proof}
We now move to the last step towards proving Proposition \ref{prop:MALA}: showing that for $\beta >2$, Assumption \ref{ass:compactSpaceOrLyapunov} cannot hold, as the necessary condition \eqref{eq:intAto1} from Lemma \ref{lem:rto0} is violated when $\beta >2$.
\begin{lemma}
\label{lem:beta>2}
    Let $\beta>2$. Then,
    \begin{equation*}
        \lim_{|x|\to\infty}\int_Sa(x,y)dy<1.
    \end{equation*}
\end{lemma}

\begin{proof}
    If $\beta>2$, as $|x|\to\infty$ the leading term in $g(x,t)$ is $-\frac{\varepsilon}{8}(\gamma\beta)^2\left(|y(x,t)|^{2\beta-2}-|x|^{2\beta-2}\right)$. 
    Moreover, from \eqref{eq:y2} we see that, as $|x|\to\infty$,
    \begin{equation*}
        |y(x,t)|^2=\left(\frac{\varepsilon\gamma\beta}{2}\right)^2|x|^{2\beta-2}+o\left(|x|^{2\beta-2}\right).
    \end{equation*}
    This implies that $|x|^{2\beta-2}=o\left(|y(x,t)|^{2\beta-2}\right)$. Consequently, the liminf in \eqref{MALAprooflimitge0} becomes
    \begin{align*}
        \liminf_{|x|\to\infty} g(x,t) = \liminf _{|x| \to \infty} \left( -\frac{\varepsilon}{8}(\gamma\beta)^2|y(x,t)|^{2\beta-2} \right)=-\infty.
    \end{align*}
    By Lemma~\ref{lem:equivalentForIntA} we obtain that $\int_Sa(x,y)dy\neq1$. From the definition of $a(x,y)$ as a transition kernel, it cannot hold that $\int_Sa(x,y)dy>1$. This completes the proof.
\end{proof}

\end{appendix}

%
%

\begin{acks}[Acknowledgments]
We thank Prof A.\ Budhiraja (UNC Chapel Hill) for insightful comments and for pointing us to the paper \cite{BD03} in connection with the discussion in Section \ref{sec:conjecture}. We also thank the AE and anonymous reviewers for very thorough and helpful feedback on the first version of this manuscript. Their comments and questions have helped improve the clarity of the paper as well as correct some errors in the initial proofs. 
\end{acks}
\begin{funding}
FM and PN were supported by the Swedish e-Science Research Centre (SeRC). PN was also supported by Wallenberg AI, Autonomous Systems and Software Program (WASP) funded by the Knut and Alice Wallenberg Foundation, and by the Swedish Research Council (VR-2018-07050, VR-2023-03484).
\end{funding}

\bibliographystyle{imsart-number}

\bibliography{bibliography}       


\end{document}